\documentclass[UTF8,11pt]{ctexart}
\usepackage[a4paper,margin=2.8cm]{geometry}
\usepackage{amsmath,amssymb,amsthm,mathtools,mathrsfs,bm}
\usepackage{iftex}
\ifPDFTeX
  \usepackage[T1]{fontenc}
\fi
\IfFileExists{newtxtext.sty}{\usepackage{newtxtext}}{}
\IfFileExists{newtxmath.sty}{\usepackage{newtxmath}}{}
\usepackage{microtype}
\usepackage{enumitem}
\usepackage{setspace}
\usepackage{hyperref}
\usepackage[nameinlink,noabbrev]{cleveref}
\usepackage{aliascnt}
\usepackage{fancyhdr}
\usepackage{lastpage}
\hypersetup{colorlinks=true,linkcolor=blue,citecolor=blue,urlcolor=blue}

\numberwithin{equation}{section}
\allowdisplaybreaks

\theoremstyle{plain}
\newtheorem{theorem}{Theorem}[section]

\newaliascnt{lemma}{theorem}
\newtheorem{lemma}[lemma]{Lemma}
\aliascntresetthe{lemma}

\newaliascnt{proposition}{theorem}
\newtheorem{proposition}[proposition]{Proposition}
\aliascntresetthe{proposition}

\newaliascnt{corollary}{theorem}
\newtheorem{corollary}[corollary]{Corollary}
\aliascntresetthe{corollary}

\theoremstyle{definition}
\newaliascnt{definition}{theorem}
\newtheorem{definition}[definition]{Definition}
\aliascntresetthe{definition}

\newaliascnt{example}{theorem}
\newtheorem{example}[example]{Example}
\aliascntresetthe{example}

\theoremstyle{remark}
\newaliascnt{remark}{theorem}
\newtheorem{remark}[remark]{Remark}
\aliascntresetthe{remark}

\theoremstyle{definition}
\newaliascnt{question}{theorem}
\newtheorem{question}[question]{Question}
\aliascntresetthe{question}

\Crefname{question}{question}{questions}
\Crefname{question}{Question}{Questions}

\crefname{theorem}{theorem}{theorems}
\Crefname{theorem}{Theorem}{Theorems}
\crefname{lemma}{lemma}{lemmas}
\Crefname{lemma}{Lemma}{Lemmas}
\crefname{proposition}{proposition}{propositions}
\Crefname{proposition}{Proposition}{Propositions}
\crefname{corollary}{corollary}{corollaries}
\Crefname{corollary}{Corollary}{Corollaries}
\crefname{definition}{definition}{definitions}
\Crefname{definition}{Definition}{Definitions}
\crefname{example}{example}{examples}
\Crefname{example}{Example}{Examples}
\crefname{remark}{remark}{remarks}
\Crefname{remark}{Remark}{Remarks}

\newcommand{\R}{\mathbb{R}}
\newcommand{\Z}{\mathbb{Z}}

\newcommand{\supp}{\operatorname{supp}}

\newcommand{\papertitle}{Spectral eigenvalue set of self‑similar measures associated\\ with product‑form Hadamard triples}
\title{\papertitle}
\date{}

\newcommand{\paperauthors}{Xin Yang$^{1}$\quad Wei-Jie Wang$^{2,*}$}

\newcommand{\paperkeywords}{Spectral measure; product‑form Hadamard triples; Spectral eigenvalue; Beurling dimension.}

\newcommand{\paperaddressone}{School of Mathematics and Statistics, Huazhong University of Science and Technology, Wuhan 430074, China}
\newcommand{\paperemailone}{xyang@hust.edu.cn}
\newcommand{\paperaddresstwo}{Hubei Key Laboratory of Mathematical Sciences, College of Mathematics and Statistics, Central China Normal University, Wuhan, Hubei 430079, China}
\newcommand{\paperemailtwo}{wwjmath@163.com}

\begin{document}

\begin{center}
{\Large\bfseries \begin{tabular}{@{}c@{}}\papertitle\end{tabular}\par}
\vspace{0.9em}
{\large \paperauthors\par}
\end{center}
\vspace{0.7em}

\begin{abstract}
Previously, An \cite{AL01} showed that the self-similar measure $\mu$ generated by a product-form Hadamard triple is a spectral measure. In this paper, we study its spectral eigenvalue problem. A set $A\subset\mathbb R$ is called a spectral eigenvalue set of $\mu$ if there exists a spectrum $\Lambda$ of $\mu$ such that
$a\Lambda$ is a spectrum of $\mu$ for every $a\in A$. We introduce the Product-form Hadamard multiplier set $\mathcal{T}_*$, and prove that for any $s\in [0,\frac{\log \#\mathcal{D}}{\log N}]$, the spectral eigensubspace 
$$V^{(s)}(\mu_{N,\mathcal{D}},\mathcal{T}_*):=\{\Lambda :t \Lambda \text{ is a spectrum of }\mu \text{ for all }t \in\mathcal{T}_* \text{ and } \dim_{Be}(\Lambda)=s\}$$
has the cardinality of the continuum. This result allows us to show that for the four-digit self-similar measures, a real number $t$ is a spectral eigenvalue if and only if $t \in \left\{\frac{u}{v}:u,v\in 2\mathbb{Z}+1\right\}$.
And for any subset $S$ of $\mathbb{R}$ is a spectral eigenvalue set if and only if $S \subset t^{-1} (2\mathbb{Z}+1)$ for some $t\in 2\mathbb{Z}+1$.
\end{abstract}

\begingroup
\renewcommand{\thefootnote}{}
\footnotetext{%
\par
Corresponding author $*$.}
\endgroup

\section{Introduction}

A Borel probability measure $\mu$ on $\mathbb{R}^d$ is called a spectral measure if we can find a countable set $\Lambda \subset \mathbb{R}^d$ such that the set of exponential functions $E(\Lambda):=\{e^{2\pi i \lambda\cdot x}:\lambda \in \Lambda\}$ forms an orthogonal basis for $L^2(\mu)$. If such $\Lambda$ exists, then $\Lambda$ is called a spectrum for $\mu$.

The study of spectra has a long time and has received much more attention from
the famous spectral set conjecture raised by Fuglede \cite{Fuglede01} in 1974. According to the result of He, Lau and Lai \cite{HLL01}, to study this issue, one just need to focus on measures $\mu$ of pure type, i.e., $\mu$ is either discrete with finite support, or singularly continuous or absolutely continuous with respect to Lebesgue measure. For a discrete measure $\mu$, the
spectrality of $\mu$ is closely related to the integer tiles; For an absolutely continuous measure $\mu$, Dutkay and Lai \cite{DL02} showed that an absolutely continuous spectral measure
is a normalized Lebesgue measure restricted on its support. For a singularly continuous measure $\mu$, the spectrality has a big difference from the other two cases and our understanding of singular continuous measures is far less profound than that of the other two cases. 

In recent decades, the intersection of fractal geometry with various other branches of mathematics has been very active, such as harmonic analysis, wavelet analysis, PDEs, and so on. Recall that given a finite collection of maps
$$\phi_d(x)=\frac{1}{N}(x+d),$$
where $d \in \mathcal{D}$, the (equal-weight) self-similar measure $\mu=\mu_{N,\mathcal{D}}$ is the unique probability measure such that
$$\mu(E)=\frac{1}{\#D}\sum_{d \in \mathcal{D}} \mu(\phi_d^{-1}(E)), \text{ for any Borel set } E .$$
The attractor is the unique non-empty compact set $K=K(N,\mathcal{D})$ satisfying the identity $K=\bigcup_{d\in \mathcal{D}}\phi_i(K)$. And the measure $\mu$ can be written as an infinite convolution of discrete measures 
\begin{equation}\label{eq:1.1}
\mu_{N,\mathcal{D}}=\delta_{N^{-1}\mathcal{D}} * \delta_{N^{-2}\mathcal{D}} * \cdots
\end{equation}
where the dirac measure $\delta_{\mathcal{D}}=\frac{1}{\#\mathcal{D}}\sum_{d\in \mathcal{D}} \delta_d$. 
There is a crucial concept underlying the study of spectrality for self-similar measures. 
\begin{definition}
Let $N \ge 2$ be an integer, and let $D, L \subset \mathbb{Z}$ be two finite sets with $\#D=\#L$
The triple $(N, D, L)$ is called a \emph{Hadamard triple} if the matrix
\begin{equation*}
    H := \frac{1}{\sqrt{m}} \left( e^{-2\pi i \, l d / N} \right)_{l \in L, \, d \in D}
\end{equation*}
is unitary, i.e., $H^*\cdot H = I$. 
Equivalently, the following orthogonality condition is satisfied:
\begin{equation*}
    \widehat{\delta}_{N^{-1}D}(l_1-l_2)=\frac{1}{\#D}\sum_{d \in D} e^{-2\pi i (l_1 - l_2) d / N} = 0, \qquad \forall \, l_1, l_2 \in L.
\end{equation*}
\end{definition}

In 2002, {\L}aba and Wang \cite{LW02} proved that if $(N,D,L)$ is a Hadamard triple, then the associated self-similar measure $\mu_{N,D}$ is spectral. This result was later extended to higher-dimensional self-affine measures by Dutkay, Haussermann and Lai \cite{DDHL02}.

\subsection{Product-form Hadamard triple}
However, the Hadamard triple condition is only a sufficient condition and is not necessary for the spectrality of self-similar measures.

\begin{example}
As a simplest example, Consider
$N=2, D=\{0,2\}$. Then the self-similar probability measure $\mu=\mu_{2,D}$ is precisely the normalized Lebesgue measure on $[0,2]$. It follows that $\mu$ is spectral and
\[
\Lambda=\frac12\mathbb Z
\]
is a spectrum for $\mu$. there is no integer set $L\subset\mathbb Z$ such that
$(2,\{0,2\},L)$
is a Hadamard triple.
\end{example}

Motivated by the fact that ordinary Hadamard triples do not cover all spectral self-similar measures, An and Lai \cite{AL01} introduced the notion of product-form Hadamard triples.

\begin{definition}[Product-form Hadamard triple]
    Let $N \geq 2$ be an integer and $\mathcal{A}=\{a_s:s=0,1,\ldots,n-1\}$ be a subset of integers and for each $s$, we have $\mathcal{B}_s$ as another subset of integers. We say that $\mathcal{D}$ is a \textit{product-form digit set} if there exists $r \geq 0$ such that 
$$\mathcal{D}=\mathcal{D}_r=\bigcup_{s=0}^{n-1} (a_s+N^r\mathcal{B}_s)$$
and we have the following conditions for $\mathcal{A},\mathcal{B}_s,\mathcal{L}_1,\mathcal{L}_2$:\\

\noindent (i) $(N,\mathcal{A},\mathcal{L}_1)$ and $(N,\mathcal{B}_s,\mathcal{L}_2)$ are Hadamard triples for all $s\in \{0,1,\ldots,n-1\}$.

\noindent (ii) $(N,\mathcal{A} \oplus\mathcal{B}_s,\mathcal{L}_1\oplus\mathcal{L}_2)$ are Hadamard triples for all $s\in \{0,1,\ldots,n-1\}$.
\end{definition}

An and Lai \cite{AL01} proved that every self-similar measure generated by a product-form Hadamard triple is spectral. Their result provides a larger class of spectral self-similar measures than those arising from ordinary Hadamard triples. In particular, some product-form constructions yield spectral self-similar measures whose digit sets do not admit any ordinary Hadamard triple structure.

Besides proving the spectrality of a measure, it is also natural to study the structure of its spectra. This direction goes back to the fundamental example of Jorgensen and Pedersen \cite{JS01}, who constructed the first singular non-atomic spectral measure, namely the middle fourth Cantor measure $\mu_{4,\{0,2\}}$, and gave the standard spectrum
\[
\Lambda_0=
\left\{
\sum_{j=0}^{n}4^j a_j:
a_j\in\{0,1\},\ n\geq 0
\right\}.
\]
Subsequent works of Strichartz \cite{Strichartz01,Strichartz02} and Dutkay, Han and Sun \cite{Ervin Dutkay2014} revealed that different spectra of the same singular measure may lead to quite different Fourier analytic behavior. For example, although both $\Lambda_0$ and $17\Lambda_0$ are spectra of $\mu_{4,\{0,2\}}$, the corresponding mock Fourier series may have different convergence and divergence properties for certain continuous functions.

These phenomena motivated the study of spectral eigenvalue problems \cite{Dai2016,DutkayHaussermann2016,FuHe2017,FuHeWen2018,HeTangWu2019}. For a spectral measure $\mu$ and a spectrum $\Lambda$, one asks for which scaling factors $t\in\mathbb R$ the dilated set $t\Lambda$ is still a spectrum of $\mu$. Such numbers are called spectral eigenvalues. More generally, a set $A\subset\mathbb R$ is called a spectral eigenvalue set of $\mu$ if there exists a spectrum $\Lambda$ of $\mu$ such that
$a\Lambda$ is a spectrum of $\mu$ for every $a\in A$. Such a spectrum $\Lambda$ is called an eigenvalue-spectrum for $(\mu, A)$. The collection of all eigenvalue-spectra is denoted by
$$
V(\mu, A)
=
\{\Lambda:\ \Lambda \text{ is a spectrum of }\mu
\text{ and } a\Lambda \text{ is a spectrum of }\mu
\text{ for all }a\in A\},
$$
where spectra are usually considered up to translation.

For the middle fourth Cantor measure $\mu_{4,\{0,2\}}$, Deng, Fu and Kang \cite{Guotai Deng 01} showed that the set of odd integers forms a spectral eigenvalue set and that the corresponding spectral eigen-subspace has the cardinality of the continuum. Under the classical Hadamard triple assumption, Kong, Li and Wang \cite{KongLiWang2025} proved that if
\[
 A\subset \{p\in\mathbb Z:\gcd(p,q)=1\}
\]
has only finitely many prime factors, then $ A$ is a spectral eigenvalue set for $\mu_{q,D}$ and $V(\mu_{q,D}, A)$ is infinite. They further asked whether $V(\mu_{q,D}, A)$ always has the cardinality of the continuum. This question was recently answered affirmatively by Lu \cite{Lu 01}, who removed the finite-prime-factor restriction and proved that the corresponding spectral eigen-subspaces have continuum cardinality. 

It is natural to ask whether such a result remains valid in a more general setting. 
Since product-form Hadamard triples provide a broader framework for constructing 
spectral self-similar measures, we further ask whether Lu's continuum-cardinality 
result for spectral eigen-subspaces admits a corresponding extension in the 
product-form Hadamard triple.

\begin{question}
Let $(N,\mathcal D,\mathcal L_1\oplus\mathcal L_2)$ be a product-form Hadamard triple, and let $\mu_{N,\mathcal D}$ be the associated self-similar spectral measure. Suppose that $ A$ is a spectral eigenvalue set of $\mu_{N,\mathcal D}$. Does the spectral eigen-subspace
\[
V(\mu_{N,\mathcal D},A)
\]
have the cardinality of the continuum?
\end{question}

Before answering this question, we introduce the Beurling dimension, a notion of dimension that measures the size of discrete sets.
$$\dim_{Be}(\Lambda)=\limsup_{h\rightarrow \infty} \sup_{x \in \mathbb{R}} \frac{\log \#(\Lambda \cap (x-\frac{h}{2},x+\frac{h}{2}))}{\log h}.$$
He et al.\cite{Hekang 01} proved that $\dim_{Be} (\Lambda) \leq \dim_H(\supp \mu)$ holds for all spectra $\Lambda$ of self-similar spectral measures, and An  et al.\cite{ALAI01} showed that, under the Hadamard triple condition, there exist spectra with Beurling dimension zero. Naturally, this leads to the following interpolation problem.

\begin{question}
Let $(N,\mathcal D,\mathcal L_1\oplus\mathcal L_2)$ be a product-form Hadamard triple, and let $\mu_{N,\mathcal D}$ be the associated self-similar spectral measure. For $0 \leq s \leq \dim_H(\supp \mu)$, does there exists a spectrum $\Lambda$ of $\mu$ with $\dim_{Be} (\Lambda)=s$?
\end{question}

\subsection{main results}
In fact, the main objects studied in this paper are singular continuous self-similar measures, which form a rich and important class of fractal measures.

\begin{example}
Let $N=8$ and $\mathcal D=\{0,4,8,12\}$. Let
$$
\mathcal A=\{0,4\},
\qquad
\mathcal B_0=\mathcal B_1=\{0,1\},
$$
and take
$$
\mathcal L_1=\{0,1\},
\qquad
\mathcal L_2=\{0,4\}.
$$
A direct verification gives that
$$
(8,\mathcal A,\mathcal L_1), \quad
(8,\mathcal B_s,\mathcal L_2), \quad (8,\mathcal A\oplus\mathcal B_s,\mathcal L_1\oplus\mathcal L_2)
$$
are Hadamard triples for $s=0,1$. Finally,
$$
\mathcal D
=
\bigcup_{a\in\mathcal A}
\left(a+8\{0,1\}\right).
$$
Thus $\mathcal D$ is a product-form digit set. It is well known that $(8,\frac{1}{4}\mathcal{D},L)$
is a Hadamard triple for $L=\{0,2,4,6\}$.
Hence the self-similar measure
$\nu=\mu_{8,\frac{1}{4}\mathcal{D}}$
is spectral measure. Moreover, according to Lu's result \cite{Lu 01}, for any $0\leq s \leq \frac{2}{3}$, 
$$\{\Gamma \in V(\nu,\mathcal{T}_*):\dim_{Be}(\Gamma)=s\}$$ has the cardinality of the continuum. Now let $T(x)=4x$. 
Since $\mathcal D=4E$, we have $\mu=\nu(T^{-1}).$
Therefore $\mu$ is also spectral. More precisely, for any $\Gamma \in V(\nu,\mathcal{T}_*)$, 
$$
\Lambda=\frac14\Gamma
$$
is a spectrum for $\mu$. This implies that $$\{\Lambda \in V(\mu_{8,\mathcal{D}},\mathcal{T}_*):\dim_{Be}(\Lambda)=s\}$$
has the cardinality of the continuum.
\end{example}

Let the Product-form Hadamard multiplier set
\begin{align*}
\mathcal{T}_*
:=\Big\{ t\in\mathbb{Z}:\ 
&\text{for every } s\in\{0,1,\ldots,n-1\}, \text{ the triples }\\
&(N,\mathcal{A},t\mathcal{L}_1),
(N,\mathcal{B}_s,t\mathcal{L}_2),
(N,\mathcal{A}\oplus\mathcal{B}_s,t\mathcal{L}_1\oplus t\mathcal{L}_2)
\text{are Hadamard triples}
\Big\}.
\end{align*}
Our results are as follows, which answer the questions posed above.

\begin{theorem}\label{thm:1.7}
    Let $(N,\mathcal{D},\mathcal{L}_1\oplus\mathcal{L}_2)$ be a product-form Hadamard triple with $N>\#\mathcal{D}$. Then for any $s\in [0,\frac{\log \#\mathcal{D}}{\log N}]$, the spectral eigen-subspace $$V^{(s)}(\mu_{N,\mathcal{D}},\mathcal{T}_*):=\{\Lambda \in V(\mu_{N,\mathcal{D}},\mathcal{T}_*):\dim_{Be}(\Lambda)=s\}$$ has the cardinality of the continuum.
\end{theorem}

The above theorem shows that the Product-form Hadamard multiplier set defined in this paper is contained in the spectral eigenvalue set of $\mu$. A natural question is whether every spectral eigenvalue of $\mu$ must arise in this way. Four-digit spectral self-similar measures have long been an important class of objects in the study of spectral measures, and their spectrality was completely characterized by An, He and Lai \cite{AHLFour}. In what follows, we give a complete discussion of their spectral eigenvalues.

Let $\beta \geq 1$, the sets $ \mathcal O:=2\mathbb{Z}+1$ and
$$\mathcal O/\mathcal O=\left\{\frac{u}{v}:u,v\in\mathcal O\right\}.$$
Define $N=2^\beta m$ and the digit set 
$$D=\{0,a,2^\tau\ell_1,a+2^\tau\ell_2\}$$
where $m, a,\ell_1,\ell_2\in\mathcal O $ and $\tau=\beta k+r$ for some $k\in \mathbb Z_{\ge0}$ and $0<r<\beta$.

\begin{theorem}\label{thm:four-digit-application}
Let $\mu=\mu_{N,D}$ be a singular spectral self-similar measure in the normalized four-digit form above, and assume $N>4$. Then
$$
E_2(\mu)=\{t\in\mathbb R\setminus\{0\}:V(\mu,\{1,t\})\ne\emptyset\}=\mathcal O/\mathcal O.
$$
More generally, for every $S\subset \mathbb R\setminus\{0\}$,
$$
V(\mu,S)\ne\emptyset
\quad\Longleftrightarrow\quad
S\subset T^{-1}\mathcal O
\quad\text{for some } T\in\mathcal O.
$$
Whenever this condition holds, for each
$0\leq s\leq \log 4/\log N$, the set
$$
\{\Lambda\in V(\mu,S):\dim_{Be}(\Lambda)=s\}
$$
has the cardinality of the continuum.
\end{theorem}

The paper is organized as follows. In Section 2, we collect the necessary preliminaries and establish an important criterion that will be used throughout the paper. Section 3 contains several technical lemmas, which play a crucial role in the subsequent arguments. In Section 4, we prove Theorem \ref{thm:1.7}. Finally, in Section 5, we prove Theorem \ref{thm:four-digit-application}.

\section{Preliminaries}
We gather in this section the notation and several standard facts that will be used repeatedly in the proof of our main results. Let $\mu$ be a compactly supported Borel probability measure on $\mathbb R$. We use the convention
\[
\widehat{\mu}(\xi)
=
\int_{\mathbb R} e^{-2\pi i\xi x}\,d\mu(x),
\qquad \xi\in\mathbb R.
\]
The zero set of $\widehat{\mu}$ will be denoted by
\[
\mathcal Z(\widehat{\mu})
=
\{\xi\in\mathbb R:\widehat{\mu}(\xi)=0\}.
\]
This zero set provides the basic criterion for exponential orthogonality. Indeed, if
\[
E(\Lambda)=\{e^{2\pi i\lambda x}:\lambda\in\Lambda\},
\]
then a discrete set $\Lambda\subset\mathbb R$ is an orthogonal set in $L^2(\mu)$ if and only if
\[
(\Lambda-\Lambda)\setminus\{0\}
\subset
\mathcal Z(\widehat{\mu}).
\]
This criterion will be used frequently in the construction and verification of spectra.

For a finite set $E\subset\mathbb R$, we write
\[
M_E(\xi)=\widehat{\delta_E}(\xi)
=
\frac1{\#E}\sum_{e\in E}e^{-2\pi i e\xi}.
\]
By \eqref{eq:1.1}, the self-similar measure $\mu=\mu_{N,\mathcal{D}}$ satisfies the following infinite product formula
\[
\widehat{\mu}(\xi)
=
\prod_{j=1}^{\infty}
M_{\mathcal D}\left(\frac{\xi}{N^j}\right)
=
\prod_{j=1}^{\infty}
\widehat{\delta}_{N^{-j}\mathcal D}(\xi).
\]
Consequently,
\[
\mathcal Z(\widehat{\mu})
=
\bigcup_{j=1}^{\infty}
N^j\mathcal Z(M_{\mathcal D}).
\]

For each $p\geq 1$, we decompose $\mu$ as
\[
\mu=\mu_p*\mu_{>p},
\]
where
\[
\mu_p
=
\delta_{N^{-1}\mathcal D}
*
\delta_{N^{-2}\mathcal D}
*
\cdots
*
\delta_{N^{-p}\mathcal D}
\]
is the finite-level approximation and
\[
\mu_{>p}
=
\delta_{N^{-(p+1)}\mathcal D}
*
\delta_{N^{-(p+2)}\mathcal D}
*
\cdots
\]
is the tail measure. Hence
\[
\widehat{\mu}(\xi)
=
\widehat{\mu}_p(\xi)\widehat{\mu}_{>p}(\xi).
\]

Let
\[
\mathcal L=\mathcal L_1\oplus\mathcal L_2
\]
and, for $p\geq 1$, define
\[
\Gamma_p
=
\mathcal L+N\mathcal L+\cdots+N^{p-1}\mathcal L.
\]
Let $\{p_k\}_{k=1}^{\infty}$ be a sequence of positive integers. Set
\[
q_0=0,
\qquad
q_k=p_1+\cdots+p_k,\quad k\geq 1.
\]
We define a sequence of finite frequency sets inductively as follows. Let $\Lambda_0=\{0\}$ and
\[
\Lambda_{q_k}
=
\widetilde{\Gamma}_{p_1}
+
N^{q_1}\widetilde{\Gamma}_{p_2}
+
\cdots
+
N^{q_{k-1}}\widetilde{\Gamma}_{p_k},
\]
where
\[
0\in \widetilde{\Gamma}_{p_j},
\qquad
\widetilde{\Gamma}_{p_j}\equiv \Gamma_{p_j}
\pmod{N^{p_j}},
\qquad 1\leq j\leq k.
\]
Then
\[
\Lambda_{q_k}\subset \Lambda_{q_{k+1}},
\]
and we put
\begin{equation}\label{eq:2.1}
\Lambda
=
\bigcup_{k=1}^{\infty}
\Lambda_{q_k}.
\end{equation}

The following finite-level identity is the main tool for verifying spectrality later. It can be deduced from \cite[Lemma~4.2]{AL01}, but we include the proof for completeness.

\begin{lemma}\label{lem:2.1}
Let $s\in\{0,1,\ldots,n-1\}$. Then, for every $k\geq 1$ and every $\xi\in\mathbb R$,
$$
\sum_{\lambda\in\Lambda_{q_k}}
\left|\widehat{\mu}_{q_k}(\xi+\lambda)\right|^2
\left|
M_{N^{-q_k}\mathcal B_s}(\xi+\lambda)
\right|^2
=
\frac1n
\sum_{i=0}^{n-1}
\left|M_{\mathcal B_i}(\xi)\right|^2.
$$
\end{lemma}

\begin{proof}
We argue by induction on $k$. When $k=1$, we have $q_1=p_1$ and
$$
\Lambda_{q_1}
=
\widetilde{\Gamma}_{p_1}
\equiv
\Gamma_{p_1}
\pmod{N^{p_1}}.
$$
Thus the desired identity follows directly from \cite[Lemma~4.2]{AL01}.

Assume that the assertion holds for $k-1$. We prove it for $k$. Since
$$
\Lambda_{q_k}
=
\Lambda_{q_{k-1}}
+
N^{q_{k-1}}\widetilde{\Gamma}_{p_k},
$$
every element of $\Lambda_{q_k}$ can be written in the form
$$
\lambda+N^{q_{k-1}}\ell,
\qquad
\lambda\in\Lambda_{q_{k-1}},
\quad
\ell\in\widetilde{\Gamma}_{p_k}.
$$
Using the finite-level decomposition of the Fourier transform, we have
$$
\widehat{\mu}_{q_k}
\left(\xi+\lambda+N^{q_{k-1}}\ell\right)
=
\widehat{\mu}_{q_{k-1}}(\xi+\lambda)
\widehat{\mu}_{p_k}
\left(
\frac{\xi+\lambda}{N^{q_{k-1}}}+\ell
\right).
$$
Moreover,
$$
M_{N^{-q_k}\mathcal B_s}
\left(
\xi+\lambda+N^{q_{k-1}}\ell
\right)
=
M_{N^{-p_k}\mathcal B_s}
\left(
\frac{\xi+\lambda}{N^{q_{k-1}}}+\ell
\right).
$$
Therefore, by \cite[Lemma~4.2]{AL01},
\begin{align*}
&\sum_{\eta\in\Lambda_{q_k}}
\left|\widehat{\mu}_{q_k}(\xi+\eta)\right|^2
\left|M_{N^{-q_k}\mathcal B_s}(\xi+\eta)\right|^2\\
=&
\sum_{\lambda\in\Lambda_{q_{k-1}}}
\left|\widehat{\mu}_{q_{k-1}}(\xi+\lambda)\right|^2
\sum_{\ell\in\widetilde{\Gamma}_{p_k}}
\left|
\widehat{\mu}_{p_k}
\left(
\frac{\xi+\lambda}{N^{q_{k-1}}}+\ell
\right)
\right|^2
\left|
M_{N^{-p_k}\mathcal B_s}
\left(
\frac{\xi+\lambda}{N^{q_{k-1}}}+\ell
\right)
\right|^2
\\
=&
\sum_{\lambda\in\Lambda_{q_{k-1}}}
\left|\widehat{\mu}_{q_{k-1}}(\xi+\lambda)\right|^2
\frac1n
\sum_{j=0}^{n-1}
\left|
M_{\mathcal B_j}
\left(
\frac{\xi+\lambda}{N^{q_{k-1}}}
\right)
\right|^2
\\
=&
\frac1n
\sum_{j=0}^{n-1}
\sum_{\lambda\in\Lambda_{q_{k-1}}}
\left|\widehat{\mu}_{q_{k-1}}(\xi+\lambda)\right|^2
\left|
M_{N^{-q_{k-1}}\mathcal B_j}(\xi+\lambda)
\right|^2.
\end{align*}
Applying the induction hypothesis to each $j=0,1,\ldots,n-1$, we get
\begin{align*}
\sum_{\eta\in\Lambda_{q_k}}
\left|\widehat{\mu}_{q_k}(\xi+\eta)\right|^2
\left|M_{N^{-q_k}\mathcal B_s}(\xi+\eta)\right|^2=
\frac1n
\sum_{j=0}^{n-1}
\left(
\frac1n
\sum_{i=0}^{n-1}
\left|M_{\mathcal B_i}(\xi)\right|^2
\right)=
\frac1n
\sum_{i=0}^{n-1}
\left|M_{\mathcal B_i}(\xi)\right|^2.
\end{align*}
This completes the proof.
\end{proof}

The following proposition is a key ingredient in the verification of spectrality. It follows from Lemma~\ref{lem:2.1} and \cite[Proposition~4.3]{AL01}; hence its proof is omitted.

\begin{proposition}\label{prop:2.2}
For any $\Lambda$ defined by \eqref{eq:2.1}, the following assertions hold.

\noindent\emph{(i)} For every $\xi\in\mathbb R$,
$$
\sum_{\lambda\in\Lambda}
\left|\widehat{\mu}(\xi+\lambda)\right|^2
\leq
\frac1n
\sum_{s=0}^{n-1}
\left|M_{\mathcal B_s}(\xi)\right|^2.
$$

\noindent\emph{(ii)} Suppose that there exists a constant $c>0$ such that, for every $k\geq 1$, every $\lambda\in\Lambda_{q_k}$, and every $\xi\in[0,1]$,
$$
\left|\widehat{\mu}_{>q_k}(\xi+\lambda)\right|^2
\geq
\frac{c}{n}
\sum_{s=0}^{n-1}
\left|
M_{N^{-q_k}\mathcal B_s}(\xi+\lambda)
\right|^2.
$$
Then, for every $\xi\in[0,1]$,
$$
\sum_{\lambda\in\Lambda}
\left|\widehat{\mu}(\xi+\lambda)\right|^2
=
\frac1n
\sum_{s=0}^{n-1}
\left|M_{\mathcal B_s}(\xi)\right|^2.
$$
\end{proposition}

\begin{remark}
    In the following sections, we may assume that $r = 1$, which does not affect our results; see \cite{AL01} for the justification.
\end{remark}

\section{Two technical Lemmas}
In this section, we develop two technical ingredients that will play a crucial role in the construction of uncountably many spectra in the next section. Following the notation of \cite{AL01}, we introduce a convention that allows us to avoid the possible obstruction caused by the zero point. 

Let
$$
M_{\mathcal B_s}(\xi)
=
\frac{1}{\#\mathcal B_s}
P_{\mathcal B_s}(e(-\xi)),
\qquad
P_{\mathcal B_s}(z)
=
\sum_{b\in\mathcal B_s}z^b,
$$
where \(e(x)=e^{2\pi i x}\). Then
$M_{\mathcal B_s}(\theta)=0$ if and only if
$P_{\mathcal B_s}(e(-\theta))=0.$
We denote the common zero set on the unit circle by
$$
\mathcal Z=\left\{
e(-\theta)\in\mathbb T:
P_{\mathcal B_s}(e(-\theta))=0
\text{ for all }s=0,1,\ldots,n-1
\right\}.$$
For every $e(-\theta) \in \mathcal{Z}$, let $F_{\theta}(z)\in\mathbb Z[x]$ be the minimal polynomial of $e(-\theta)$. Since $P_{\mathcal B_s}\in\mathbb Z[x]$ and $e(-\theta)$ is a zero of $P_{\mathcal B_s}$, we have
$$
F_{\theta}(x)\mid P_{\mathcal B_s}(x),
\qquad s=0,1,\ldots,n-1.
$$
Define $k_{\theta,s}=\max\{k:F_{\theta}^k(z) \mid P_{\mathcal{B}_s}(z)\}$ and 
$$k_{\theta}=\min\{k_{\theta,s}:s=0,1,\ldots,n-1\}.$$
From the preceding construction, the polynomial
$$F(x)=\prod_{e(-\theta) \in \mathcal{Z}} F_{\theta}^{k_{\theta}}(x).$$
divides each $P_{\mathcal B_s}(x)$ for $s=0,1,\ldots,n-1$. 
Therefore, for each \(s\), we may write
$$P_{\mathcal B_s}(x)=F(x)Q_s(x),$$
where $Q_s(x)\in\mathbb Z[x]$. Moreover, the polynomials \(Q_s\) satisfy
$$
Q(\xi)
:=
\frac1n
\sum_{s=0}^{n-1}
\left|Q_s(e(-\xi))\right|^2
>0,
\qquad \xi\in\mathbb R.
$$

Consequently,
$$
M_{\mathcal{B}_s}(\xi)=\frac{1}{\#\mathcal{B}_s} F(e(-\xi))Q_s(e(-\xi)),
\qquad s=0,1,\ldots,n-1.
$$
Since
$$
\mathcal D
=
\bigcup_{s=0}^{n-1}(a_s+N\mathcal B_s),
$$
and the sets \(\mathcal B_s\) have the same cardinality, we obtain
$$\widehat{\delta}_{N^{-1}\mathcal{D}}(\xi)=\frac{1}{n}\sum_{s=0}^{n-1} e^{-2\pi i \frac{a_s}{N}\xi} M_{\mathcal{B}_s}(\xi).$$

\begin{lemma}\label{lem:3.1}
For any $t \in \mathcal{T}_*$, we have
\begin{equation}\label{eq:3.1}
    \sum_{l=0}^{N-1}\lvert\frac{1}{n}\sum_{s=0}^{n-1} e^{-2\pi i \frac{a_s}{N}(x+tl)} Q_s(e(-x)) \rvert^2=\frac{N}{n} \frac{1}{n}  \sum_{s=0}^{n-1}\lvert  Q_s(e(-x))  \rvert^2 
\end{equation}
and 
\begin{equation}\label{eq:3.2}
\sum_{l=0}^{N-1}\lvert\frac{1}{n}\sum_{s=0}^{n-1} e^{-2\pi i \frac{a_s}{N}(x+tl)} Q_s(e(-x)) \rvert^2\frac{1}{n} \sum_{r=0}^{n-1} \lvert M_{\frac{\mathcal{B}_r}{N}} (x+tl)\rvert^2 =\frac{N}{n\#\mathcal{B}_s} \frac{1}{n}\sum_{s=0}^{n-1}\lvert  Q_s(e(-x))  \rvert^2 .
\end{equation}

\end{lemma}
\begin{proof} Since $(N,\mathcal{A},t\mathcal{L}_1)$ is a Hadamard triple and
$$\sum_{l=0}^{N-1} e^{2\pi i \frac{a_{s'}-a_s}{N} tl}=\begin{cases} 
  N, & \text{if } N\mid t(a_{s'}-a_s),\\
  0, & \text{if } N\nmid t(a_{s'}-a_s).
\end{cases}$$
Then we have
    \begin{align*}
        & \sum_{l=0}^{N-1}\lvert\frac{1}{n}\sum_{s=0}^{n-1} e^{-2\pi i \frac{a_s}{N}(x+tl)} Q_{s}(e(-x)) \rvert^2 \\
         =& \frac{1}{n^2}\sum_{s=0}^{n-1}\sum_{s'=0}^{n-1}e^{2\pi i \frac{a_s-a_{s'}}{N}x}Q_s(e(-x))\overline{Q_{s'}(e(-x))}\sum_{l=0}^{N-1}  e^{2\pi i \frac{a_s-a_{s'}}{N}tl} \\
         = &\frac{N}{n} \frac{1}{n} \sum_{s=0}^{n-1} \lvert Q_s(e(-x))\rvert^2.
    \end{align*}
Since $(N,\mathcal{A}\oplus \mathcal{B}_s,t\mathcal{L})$ is a Hadamard triple and
$$\sum_{l=0}^{N-1} e^{2\pi i \frac{a_{s'}-a_s+b-b'}{N}tl}=\begin{cases} 
  N, & \text{if } N\mid t(a_{s'}-a_s+b-b'),\\
  0, & \text{if } N\nmid t(a_{s'}-a_s+b-b').
\end{cases}$$
Then we have 
\begin{align*}
  &\sum_{l=0}^{N-1} e^{2\pi i \frac{a_s-a_{s'}}{N}tl}M_{\frac{\mathcal{B}_r}{N}} (x+tl)\overline{M_{\frac{\mathcal{B}_r}{N}} (x+tl)} \\
  =& \frac{1}{\# \mathcal{B}_r^2}\sum_{b\in \mathcal{B}_r}\sum_{b'\in \mathcal{B}_r}e^{2\pi i \frac{b-b'}{N}x}\sum_{l=0}^{N-1} e^{2\pi i \frac{a_{s'}-a_s+b-b'}{N}tl}=\frac{N}{\# \mathcal{B}_r}.
\end{align*}
Hence 
 \begin{align*}
&\sum_{l=0}^{N-1}\lvert\frac{1}{n}\sum_{s=0}^{n-1} e^{-2\pi i \frac{a_s}{N}(x+tl)} Q_s(e(-x)) \rvert^2\frac{1}{n} \sum_{r=0}^{n-1} \lvert M_{\frac{\mathcal{B}_r}{N}} (x+tl)\rvert^2 \\
=&\frac{1}{n} \sum_{r=0}^{n-1}\sum_{l=0}^{N-1} \lvert \frac{1}{n}\sum_{s=0}^{n-1} e^{-2\pi i \frac{a_s}{N}(x+tl)} Q_s(e(-x)) \rvert^2 \lvert M_{\frac{\mathcal{B}_r}{N}} (x+tl)\rvert^2\\
=&\frac{1}{n^3} \sum_{r=0}^{n-1} \sum_{s=0}^{n-1}\sum_{s'=0}^{n-1}e^{2\pi i \frac{a_s-a_{s'}}{N}x}Q_s(e(-x))\overline{Q_{s'}(e(-x))}\sum_{l=0}^{N-1} e^{2\pi i \frac{a_s-a_{s'}}{N}tl}M_{\frac{\mathcal{B}_r}{N}} (x+tl)\overline{M_{\frac{\mathcal{B}_r}{N}} (x+tl)}\\
=& \frac{N}{n\#\mathcal{B}_s} \frac{1}{n} \sum_{s=0}^{n-1} \lvert M_{\mathcal{B}_s} (x)\rvert^2.
 \end{align*}
This completes the proof of the theorem.

\end{proof}

\begin{corollary}\label{cor:3.2}
For any $t \in \mathcal{T}_*$, we have
\begin{equation}\label{eq:3.3}
    \sum_{l=0}^{N-1} \lvert \widehat{\delta}_{N^{-1}\mathcal{D}}(x+tl) \rvert^2=\frac{N}{n} \frac{1}{n} \sum_{s=0}^{n-1} \lvert M_{\mathcal{B}_s} (x)\rvert^2
\end{equation}
and
\begin{equation}\label{eq:3.4}
    \sum_{l=0}^{N-1} \lvert \widehat{\delta}_{N^{-1}\mathcal{D}}(x+tl) \rvert^2\cdot\frac{1}{n} \sum_{s=0}^{n-1} \lvert M_{\frac{\mathcal{B}_s}{N}} (x+tl)\rvert^2=\frac{N}{n\#\mathcal{B}_s} \frac{1}{n} \sum_{s=0}^{n-1} \lvert M_{\mathcal{B}_s} (x)\rvert^2.
\end{equation}
\end{corollary}
\begin{proof} 
    Multiplying both sides of \eqref{eq:3.1} and \eqref{eq:3.2} in Lemma \ref{lem:3.1} by $\lvert \frac{1}{\#\mathcal{B}_s} F(e(-x)) \rvert^2$ yields the desired conclusion.
\end{proof}

\begin{lemma}\label{lem:3.3}
   For any $t \in \mathcal{T_*}$, there exists a $d_t>0$, for all $x\in[0,t]$, there exist two distinct integers 
    $$l_1,l_2 \in \{0,1,\ldots,N-1\}$$ such that
$$\lvert\frac{1}{n}\sum_{s=0}^{n-1} e^{-2\pi i \frac{a_s}{N}(x+tl_i)} Q_s(e(-x)) \rvert^2\frac{1}{n} \sum_{s=0}^{n-1} \lvert M_{\frac{\mathcal{B}_s}{N}} (x+tl_i)\rvert^2 \geq \frac{d_t}{n} \sum_{s=0}^{n-1}\lvert  Q_s(e(-x))  \rvert^2 $$
for $i=1,2$.
\end{lemma}
\begin{proof}
Assume for contradiction that there exists $t \in \mathcal{T_*}$ such that for every $d>0$, one can find $x_d \in [0,
t]$ for which at most one index $l \in \{0,1,\ldots,N-1\}$ satisfies
$$\lvert\frac{1}{n}\sum_{s=0}^{n-1} e^{-2\pi i \frac{a_s}{N}(x_d+tl)} Q_s(e(-x_d)) \rvert^2\frac{1}{n}\sum_{s=0}^{n-1} \lvert M_{\frac{\mathcal{B}_s}{N}} (x_d+tl)\rvert^2 > d .$$
We now distinguish two possible situations. For each $d>0$, there exists a $x_d \in [0,t]$ such that either 
$$\lvert\frac{1}{n}\sum_{s=0}^{n-1} e^{-2\pi i \frac{a_s}{N}(x_d+tl)} Q_s(e(-x_d)) \rvert^2 \frac{1}{n}\sum_{s=0}^{n-1} \lvert M_{\frac{\mathcal{B}_s}{N}} (x_d+tl)\rvert^2<d \text{ for all }l\in \{0,1,\ldots,N-1\}$$
or there is an index $l_d \in \{0,1,\ldots, N-1 \}$ for which
$$\lvert\frac{1}{n}\sum_{s=0}^{n-1} e^{-2\pi i \frac{a_s}{N}(x_d+tl_d)} Q_s(e(-x_d)) \rvert^2 \frac{1}{n}\sum_{s=0}^{n-1} \lvert M_{\frac{\mathcal{B}_s}{N}} (x_d+tl_d)\rvert^2>d,$$
while
$$\lvert\frac{1}{n}\sum_{s=0}^{n-1} e^{-2\pi i \frac{a_s}{N}(x_d+tl)} Q_s(e(-x_d)) \rvert^2 \frac{1}{n}\sum_{s=0}^{n-1} \lvert M_{\frac{\mathcal{B}_s}{N}} (x_d+tl)\rvert^2<d \text{ for all } l\in \{0,1,\ldots,N-1\} \backslash \{l_d\}.$$\\

\textbf{Case I.} There exists a monotonically decreasing subsequence $\{d_k\}$ tending to 0 such that, for each $k$,  one can find $x_k \in [0,t]$, 
$$\lvert\frac{1}{n}\sum_{s=0}^{n-1} e^{-2\pi i \frac{a_s}{N}(x_k+tl)} Q_s(e(-x_k)) \rvert^2 \frac{1}{n}\sum_{s=0}^{n-1} \lvert M_{\frac{\mathcal{B}_s}{N}} (x_d+tl)\rvert^2<d_k \text{ for any } l\in \{0,1,\ldots,N-1\}.$$
Since $\{x_k\} \subset [0,1]$, then there exists a subsequence $\{x_{n_k} \}\subset \{x_k\}$ such that $x_{n_k} \rightarrow x \in [0,t]$ and 
$$\lvert\frac{1}{n}\sum_{s=0}^{n-1} e^{-2\pi i \frac{a_s}{N}(x+tl)} Q_s(e(-x)) \rvert^2 \frac{1}{n}\sum_{s=0}^{n-1} \lvert M_{\frac{\mathcal{B}_s}{N}} (x+tl)\rvert^2=0 \text{ for any } l\in \{0,1,\ldots,N-1\}.$$
Combining this with \eqref{eq:3.2} yields that
\begin{align*}
\sum_{s=0}^{n-1}\lvert Q_s(e(-x))\rvert^2
=0,
\end{align*}
which leads to a contradiction.\\

\textbf{Case II.} There exists a $d_0>0$ such that for any $d<d_0$, there exists a $x_d \in [0,t]$ and $l_d \in \{0,1,\ldots,N-1\}$, we have 
$$\lvert\frac{1}{n}\sum_{s=0}^{n-1} e^{-2\pi i \frac{a_s}{N}(x_d+tl_d)} Q_s(e(-x_d)) \rvert^2 \frac{1}{n}\sum_{s=0}^{n-1} \lvert M_{\frac{\mathcal{B}_s}{N}} (x_d+tl_d)\rvert^2>d,$$ 
whereas, for every $l\in \{0,1,\ldots,N-1\} \backslash \{l_d\}$
$$\lvert\frac{1}{n}\sum_{s=0}^{n-1} e^{-2\pi i \frac{a_s}{N}(x_d+tl)} Q_s(e(-x_d)) \rvert^ 2\frac{1}{n}\sum_{s=0}^{n-1} \lvert M_{\frac{\mathcal{B}_s}{N}} (x_d+tl)\rvert^2<d.$$
On the one hand, for any $d<d_0$, 
\begin{align*}
&\sum_{l=0}^{N-1}\lvert\frac{1}{n}\sum_{s=0}^{n-1} e^{-2\pi i \frac{a_s}{N}(x_d+tl)} Q_s(e(-x_d)) \rvert^ 2\frac{1}{n}\sum_{s=0}^{n-1} \lvert M_{\frac{\mathcal{B}_s}{N}} (x_d+tl)\rvert^2\\
    = & \lvert\frac{1}{n}\sum_{s=0}^{n-1} e^{-2\pi i \frac{a_s}{N}(x_d+tl_d)} Q_s(e(-x_d)) \rvert^2 \frac{1}{n}\sum_{s=0}^{n-1} \lvert M_{\frac{\mathcal{B}_s}{N}} (x_d+tl)\rvert^2\\
    +&\sum_{l\neq l_d}\lvert\frac{1}{n}\sum_{s=0}^{n-1} e^{-2\pi i \frac{a_s}{N}(x_d+tl)} Q_s(e(-x_d)) \rvert^2\frac{1}{n}\sum_{s=0}^{n-1} \lvert M_{\frac{\mathcal{B}_s}{N}} (x_d+tl)\rvert^2\\
    \leq & \lvert\frac{1}{n}\sum_{s=0}^{n-1} e^{-2\pi i \frac{a_s}{N}(x_d+tl_d)} Q_s(e(-x_d)) \rvert^2\frac{1}{n}\sum_{s=0}^{n-1} \lvert M_{\frac{\mathcal{B}_s}{N}} (x_d+tl_d)\rvert^2+(N-1) d.
\end{align*}
By the Cauchy–Schwarz inequality, we obtain that
\begin{align*}
    & \lvert\frac{1}{n}\sum_{s=0}^{n-1} e^{-2\pi i \frac{a_s}{N}(x_d+tl_d)} Q_s(e(-x_d)) \rvert^2\frac{1}{n}\sum_{s=0}^{n-1} \lvert M_{\frac{\mathcal{B}_s}{N}} (x_d+tl_d)\rvert^2\\
    \leq &\lvert\frac{1}{n}\sum_{s=0}^{n-1} e^{-2\pi i \frac{a_s}{N}(x_d+tl_d)} Q_s(e(-x_d)) \rvert^2 \leq \frac{1}{n}\sum_{s=0}^{n-1} \lvert Q_s(e(-x_d))\rvert^2.
\end{align*}
On the other hand, by \eqref{eq:3.2} in Lemma \ref{lem:3.1}, we have 
$$\frac{N}{n\#\mathcal{B}_s}\frac{1}{n} \sum_{s=0}^{n-1} \lvert Q_s(e(-x)) \rvert^2 < \frac{1}{n}\sum_{s=0}^{n-1} \lvert Q_s(e(-x_d))\rvert^2+(N-1)d.$$
This implies that 
$$d>\frac{N-n\#\mathcal{B}_s}{n(N-1)}\frac{1}{n} \sum_{s=0}^{n-1} \lvert Q_s(e(-x)) \rvert^2. $$
Since $ \frac{1}{n}\sum_{s=0}^{n-1} \lvert Q_s(e(-x)) \rvert^2$ is continuous and positive everywhere on $[0,t]$, it follows that
$$d>\frac{N-n\#\mathcal{B}_s}{n(N-1)}\min\{\frac{1}{n} \sum_{s=0}^{n-1} \lvert Q_s(e(-x)) \rvert^2: x\in[0,t]\}>0.$$
This contradicts the fact that $d \rightarrow 0$. In summary, we obtain that for any $t \in \mathcal{T_*}$, there exists a $d_t'>0$, for all $x\in[0,t]$, there exist two distinct integers 
    $$l_1,l_2 \in \{0,1,\ldots,N-1\}$$ such that
$$\lvert\frac{1}{n}\sum_{s=0}^{n-1} e^{-2\pi i \frac{a_s}{N}(x+tl_i)} Q_s(e(-x)) \rvert^2\frac{1}{n} \sum_{s=0}^{n-1} \lvert M_{\frac{\mathcal{B}_s}{N}} (x+tl_i)\rvert^2 \geq d_t' $$
for $i=1,2$. Let $d_t= d_t'\cdot {\max\{\frac{1}{n} \sum_{s=0}^{n-1} \lvert Q_s(e(-x)) \rvert^2: x\in[0,t]\}}^{-1}$, then 
$$\lvert\frac{1}{n}\sum_{s=0}^{n-1} e^{-2\pi i \frac{a_s}{N}(x+tl_i)} Q_s(e(-x)) \rvert^2\frac{1}{n} \sum_{s=0}^{n-1} \lvert M_{\frac{\mathcal{B}_s}{N}} (x+tl_i)\rvert^2 \geq \frac{d_t}{n} \sum_{s=0}^{n-1}\lvert  Q_s(e(-x))  \rvert^2.$$
This completes the proof of the theorem.

\end{proof}



As a preparation, we note the following phenomenon. The advantage lies in the boundedness of the selected integers, which greatly simplifies our later estimates. Let 
\begin{equation}\label{eq:3.5}
    \mathcal{U}=:\{x\in [0,t]: \frac{1}{n} \sum_{s=0}^{n-1} \lvert M_{\mathcal{B}_s}(x)\rvert^2>0\}
\end{equation}
\begin{lemma}\label{lem:3.4}
Let  $\mathcal{U}$  be as in \eqref{eq:3.5}. For any $t\in\mathcal{T}_*$, there exists a integer $m_t>0$ such that
$$ W_{m_t}(\widehat\mu):=\{\xi\in \mathcal{U}:\widehat\mu(\xi+tk)=0\text{ for every integer } \lvert k \rvert \leq N^{m_t}\}=\varnothing.
$$
Equivalently, for every $\xi\in \mathcal{U}$, there exists a integer $\lvert k \rvert \leq N^{m_t}$ such that $\widehat\mu(\xi+tk)\neq 0$.
\end{lemma}

\begin{proof}
Assume for contradiction that there exists $t\in\mathcal{T}_*$, for any $m\in \mathbb{N}^+$ such that
$$
        W_{m}(\widehat\mu):=\{\xi\in \mathcal{U}:\widehat\mu(\xi+tk)=0\text{ for every integer } \lvert k \rvert \leq N^{m}\}\neq \varnothing.
$$
Hence there exists $\xi_m\in W_{m}(\widehat\mu) \subset W_{m-1}(\widehat\mu)$. Fixed $m_0 \in \mathbb{N}^+$, by the monotonicity of the set $W_{m}(\widehat\mu)$, for any $m\geq m_0$, we have
$$\widehat\mu(\xi_m+tk)=0\text{ for every integer } \lvert k \rvert \leq N^{m_0}.$$
This implies that 
$$ \widehat\mu(\xi_m)=0 \text{ for all }m\geq m_0$$
From the analyticity of $\widehat{\mu}$, which has at most finitely many zeros on $[0,t]$, i.e., $\xi_m=\xi_0$ for any $m \geq m_0$. Then 
$$\xi_0 \in \bigcap_{m \geq m_0} W_{m}(\widehat\mu)= W(\widehat\mu):=\{\xi\in \mathcal{U}:\widehat\mu(\xi+tk)=0\text{ for every }k\in \mathbb{Z}\}.$$
According to \eqref{eq:3.4}, there exists a $l_0 \in \{0,1,\ldots,N-1\}$ such that 
$$\lvert \widehat{\delta}_{N^{-1}\mathcal{D}}(\xi_0+tl_0) \rvert^2\cdot\frac{1}{n} \sum_{s=0}^{n-1} \lvert M_{\frac{\mathcal{B}_s}{N}} (\xi_0+tl_0)\rvert^2 \geq \frac{1}{n\#\mathcal{B}_s} \frac{1}{n} \sum_{s=0}^{n-1} \lvert M_{\mathcal{B}_s} (\xi_0)\rvert^2>0.$$
For such $l_0$ and any $r \in \mathbb{Z}$, we get the self-similar identity
$$ \widehat\mu(\xi_0+tl_0+Ntr)=\widehat\delta_{N^{-1}D}((\xi_0+tl_0)/N)\widehat\mu((\xi+tl_0)/N+tr). $$
This implies that 
$$\widehat\mu((\xi+tl_0)/N+tr)=0 \text{ for all }r \in \mathbb{Z}, $$
which forces $\xi_1=(\xi_0+tl_0)/N\in W(\widehat\mu)$. Iterating the above procedure, we get a sequence $\{\xi_n\}$ satisfying 
$$\xi_{n+1}=(\xi_n+tl_n)/N\in W(\widehat\mu),$$
where $l_n \in \{0,1,\ldots,N-1\}$. The analyticity of $\widehat\mu$ implies that it has only finitely many zeros on $[0,t]$. Therefore, there exists $n_0$ such that for all $n \geq n_0$, each $\xi_n$ has only one offspring $\xi_{n+1}=(\xi_n+tl_n)/N\in W(\widehat\mu)$, i.e., there is only one $l_n \in \{0,1,\ldots,N-1\}$ such that
$$\lvert \widehat{\delta}_{N^{-1}\mathcal{D}}(\xi_n+tl_n) \rvert^2\cdot\frac{1}{n} \sum_{s=0}^{n-1} \lvert M_{\frac{\mathcal{B}_s}{N}} (\xi_n+tl_n)\rvert^2 \neq 0.$$
Then we get 
\begin{align*}
     \frac{N}{n\#\mathcal{B}_s} \frac{1}{n} \sum_{s=0}^{n-1} \lvert M_{\mathcal{B}_s} (\xi_n)\rvert^2=& \lvert \widehat{\delta}_{N^{-1}\mathcal{D}}(\xi_n+tl_n) \rvert^2\cdot\frac{1}{n} \sum_{s=0}^{n-1} \lvert M_{\frac{\mathcal{B}_s}{N}} (\xi_n+tl_n)\rvert^2\\
     \leq & \lvert \frac{1}{n} \sum_{s=0}^{n-1} e^{2\pi i N^{-1}a_s(\xi_n+tl_n)} M_{\mathcal{B}_s}(\xi_n)\rvert^2\\
     \leq & \frac{1}{n} \sum_{s=0}^{n-1} \lvert M_{\mathcal{B}_s} (\xi_n)\rvert^2.
\end{align*}
This is contrary to $N>\#\mathcal{D}=n\#\mathcal{B}_s$.

\end{proof}

Applying the Riesz–Markov–Kakutani representation theorem, we obtain the following lemma, which is the last result of this section.

\begin{lemma}
\label{lem:3.5}
Assume $1<\# D<N$. For every $t\in\mathcal{T}_*$ and $\xi\in \mathcal{U}$, there exists $m_t \in \mathbb{N}^+$ such that the set
$$
  Z_{m_t}(\xi):=\{k\in\Z:\widehat\mu(\xi+tk)\ne0, \lvert k \rvert \leq N^{m_t}\}
$$
has at least two elements.
\end{lemma}

\begin{proof}
 Suppose, toward a contradiction, that there exist $t_0\in\mathcal{T}_*$ and $\xi_0\in \mathcal{U}$ such that for any $m \in \mathbb{N}^+$, 
$$\#Z_{m}(\xi_0)=\#\{k\in\Z:\widehat\mu(\xi_0+t_0k)\ne0, \lvert k \rvert \leq N^{m}\}\leq 1.$$
By Lemma \ref{lem:3.4}, the set $\#Z_{m_{t_0}} (\xi_0)\geq 1$. It follows from $Z_{m}(\xi_0) \subset Z_{m+1}(\xi_0)$ that 
$$ \#Z_{m}(\xi_0)= \#Z_{m_{t_0}} (\xi_0)=1 \text{ for all }m\geq m_{t_0},$$
and there exists a unique $k_0\in \mathbb{Z}$ such that
$$ k_0 \in \bigcup_{m \in \mathbb{N}^+} Z_{m}(\xi_0)=Z(\xi_0):=\{ k \in \mathbb{Z}: \widehat\mu(\xi+tk) \neq 0\}.$$
Let $K=K(N,D)$ be the attractor of the IFS 
$$\{f_d(x)=N^{-1}(x+d):d\in D\}.$$ 
As is well known, we have 
\begin{equation*}
        \dim_H tK=\dim_H K=\frac{\log \#\mathcal{D}}{\log N}<1.
\end{equation*}
Project $tK$ modulo 1 onto the circle $\mathbb{T}$:
$$
        \pi(tK)\subset\mathbb{T}=\R/\Z.
$$
We can claim that $\pi(tK)$ has Lebesgue measure zero. Decompose $tK$ according to the unit interval as follows：
$$tK=\bigcup_{n \in \mathbb{Z}} tK \cap [n,n+1)=: \bigcup_{n \in \mathbb{Z}} K_n$$
Hence $m(K_n)=0$ for all $n \in \mathbb{Z}$. Let $T_n$ be the translation map defined by $T_n(x):=x-n$. Then $T_n(K_n) \subset [0,1)$ and $m(T_n(K_n))=T_n(K_n)=0$. Consequently, 
$$m(\pi(tK)) \leq m(\bigcup_{n \in \mathbb{Z}} T_n(E_n))=0.$$
Define a finite complex Borel measure $\nu_{\xi_0}$ on $\mathbb{T}$ by
\begin{equation*}
    \int_{\mathbb{T}}f(x)\,d\nu_{\xi_0}(x)
        =\int_{\R} f(x\bmod 1)e^{-2\pi i\xi_0 t^{-1} x}\,d\mu(t^{-1}x),
        \qquad f\in C(\mathbb{T}).
\end{equation*}
Then $\nu_{\xi_0}$ is supported on $\pi(tK)$, and its Fourier coefficients are
\begin{equation*}
        \widehat\nu_{\xi_0}(k)
        =\int_{\mathbb{T}}e^{-2\pi ikx}\,d\nu_{\xi_0}(x)
        =\int_{\R}e^{-2\pi i({\xi_0}+tk)x}\,d\mu(x)
        =\widehat\mu({\xi_0}+tk).
\end{equation*}
By assumption, all Fourier coefficients of $\nu_{\xi_0}$ vanish except possibly the one at $k_0$. The uniqueness theorem for finite measures on the torus therefore gives
\begin{equation*}
        d\nu_{\xi_0}(x)=c e^{2\pi i k_0x}\,dx,
        \qquad c=\widehat\mu({\xi_0}+tk_0)\ne0.
\end{equation*}
The measure on the right is absolutely continuous and is not the zero measure; its support has full Lebesgue measure. This contradicts the fact that $\nu_{\xi_0}$ is supported on the null set $\pi(tK)$. Consequently, the theorem is established.
\end{proof}

\section{The proof of Theorem \ref{thm:1.7}}
This section is devoted to the proof of one of our main results, Theorem \ref{thm:1.7}. Before beginning the proof, we first establish several necessary tools based on the results obtained in the preceding section.

\begin{proposition}\label{prop:4.1}
    For any $t\in \mathcal{T}_*$, there exist $d_t>0$ and $\delta_t>0$, for all $x \in [0,t]$, there exist two distinct integers 
    $$l_1,l_2 \in \{0,1,\ldots,N-1\}$$ such that
$$\lvert\widehat{\delta}_{N^{-1}\mathcal{D}}(x+y+tl_i) \rvert^2\frac{1}{n} \sum_{s=0}^{n-1}\lvert M_{\frac{\mathcal{B}_s}{N}} (x+y+tl_i)\rvert^2 \geq d_t \cdot \frac{1}{n} \sum_{s=0}^{n-1}\lvert  M_{\mathcal{B}s}(x+y)  \rvert^2 $$
and
$$\lvert\widehat{\delta}_{N^{-1}\mathcal{D}}(x+y+tl_i) \rvert^2 \geq d_t \cdot \frac{1}{n} \sum_{s=0}^{n-1}\lvert  M_{\mathcal{B}s}(x+y)  \rvert^2 $$
for all $\lvert y \rvert< \delta_t$. 
\end{proposition}
\begin{proof}
According to Lemma \ref{lem:3.3}, for any $t\in \mathcal{T}_*$, there exist $d_t>0$ and $\delta_t>0$, for all $x \in [0,t]$, there exist two distinct integers 
    $$l_1,l_2 \in \{0,1,\ldots,N-1\}$$ such that
    $$f(x):=\frac{\lvert\frac{1}{n}\sum_{s=0}^{n-1} e^{-2\pi i \frac{a_s}{N}(x+tl_i)} Q_s(e(-x)) \rvert^2\frac{1}{n} \sum_{s=0}^{n-1}\lvert M_{\frac{\mathcal{B}_s}{N}} (x+tl_i)\rvert^2}{\frac{1}{n} \sum_{s=0}^{n-1}\lvert  Q_s(e(-x))  \rvert^2} \geq d_t.$$
By continuity of $f $, for any $x \in [0,t]$, there exists a $\delta_x>0$ such that 
$$ \lvert f(x+y)- f(x) \rvert \leq \frac{d_t}{2}, \text{ for any } \lvert y \rvert < \delta_{x}$$
Since $[0,t]$ is compact set and $[0,t] \subset \bigcup_{x\in [0,t]} B(x,\delta_x)$, then
$$[0,t] \subset \bigcup_{i=1}^m B(x_i,\delta_{x_i}),\qquad x_i \in [0,t].$$
Hence there exists $\delta_t:=\min\{\delta_{x_i}: i=1,2,\ldots,m\}$, for any $x \in [0,t]$, 
$$\lvert f(x+y) \rvert \geq \lvert f(x) \rvert-\lvert f(x+y)- f(x) \rvert \geq \frac{d_t}{2}$$
for any $\lvert y \rvert < \delta_t$. Consequently, 
$$\lvert\frac{1}{n}\sum_{s=0}^{n-1} e^{-2\pi i \frac{a_s}{N}(x+y+tl_i)} Q_s(e(-x-y)) \rvert^2\frac{1}{n} \sum_{s=0}^{n-1}\lvert M_{\frac{\mathcal{B}_s}{N}} (x+y+tl_i)\rvert^2 \geq \frac{d_t}{2} \cdot\frac{1}{n} \sum_{s=0}^{n-1}\lvert  Q_s(e(-x-y))  \rvert^2 .$$
Multiplying both sides of the above inequality by $\lvert \frac{1}{\#\mathcal{B}_s} F(e(-x-y)) \rvert^2$, we obtain
$$\lvert\widehat{\delta}_{N^{-1}\mathcal{D}}(x+y+tl_i) \rvert^2 \frac{1}{n} \sum_{s=0}^{n-1}\lvert M_{\frac{\mathcal{B}_s}{N}} (x+y+tl_i)\rvert^2\geq \frac{d_t}{2} \cdot \frac{1}{n} \sum_{s=0}^{n-1}\lvert  M_{\mathcal{B}s}(x+y)  \rvert^2. $$

\end{proof}

\begin{theorem}\label{thm:4.2}
    For every $t\in \mathcal{T}_*$, there exist $\varepsilon_t>0$,  $\delta_t>0$ and $m_t \geq 1$ such that for all $x \in [0,t]$, one can find two distinct integers 
    $$k_1,k_2 \in \mathbb{Z} \cap [-N^{m_t}, N^{m_t}]$$ 
    for which 
    $$\lvert \widehat{\mu}(x+y+tk_i) \rvert^2 \geq \varepsilon_t \cdot \frac{1}{n} \sum_{s=0}^{n-1}\lvert  M_{\mathcal{B}s}(x+y)  \rvert^2, \qquad \lvert y \rvert< \delta_t$$
    for $i=1,2$. In particular, when 
$x=0$, we may take $k=0$.
\end{theorem}
\begin{proof}
    According to Proposition  \ref{prop:4.1}, we find $d_t>0$ and $\delta_t>0$, for all $x \in [0,t]$, there exist two distinct integers 
    $$l_1,l_2 \in \{0,1,\ldots,N-1\}$$ such that $$\lvert\widehat{\delta}_{N^{-1}\mathcal{D}}(x+y+tl_i) \rvert^2 \geq d_t \cdot \frac{1}{n} \sum_{s=0}^{n-1}\lvert  M_{\mathcal{B}s}(x+y)  \rvert^2 $$
for all $\lvert y \rvert< \delta_t$. And by Lemma \ref{lem:3.3}, for the two values of $l_i$ chosen above, it follows that
$$\lvert\frac{1}{n}\sum_{s=0}^{n-1} e^{-2\pi i \frac{a_s}{N}(x+tl_i)} Q_s(e(-x)) \rvert^2\frac{1}{n} \sum_{s=0}^{n-1} \lvert M_{\frac{\mathcal{B}_s}{N}} (x+tl_i)\rvert^2 \geq \frac{d_t}{n} \sum_{s=0}^{n-1}\lvert  Q_s(e(-x))  \rvert^2>0 .$$
Moreover, by Lemma \ref{lem:3.1}, for any $l\in\{0,1,\ldots,N-1\}$, 
$$ \frac{N}{n} \frac{1}{n}  \sum_{s=0}^{n-1}\lvert  Q_s(e(-x))  \rvert^2  \geq \lvert\frac{1}{n}\sum_{s=0}^{n-1} e^{-2\pi i \frac{a_s}{N}(x+tl_i)} Q_s(e(-x)) \rvert^2.$$
Hence we obtain that 
$$ \frac{1}{n} \sum_{s=0}^{n-1} \lvert M_{\frac{\mathcal{B}_s}{N}} (x+tl_i)\rvert^2  \geq d_t \frac{n}{N}.$$
Consequently, 
$$ \frac{x+tl_i}{N} \in \mathcal{U}_t:=\{ x\in[0,t]:\frac{1}{n} \sum_{s=0}^{n-1} \lvert M_{\mathcal{B}_s} (x)\rvert^2  \geq d_t \frac{n}{N} \} \subset \mathcal{U}.$$
Therefore, Lemma \ref{lem:3.5} gives two distinct integers $r_1\neq r_2$ such that
$$\widehat\mu(\frac{x+tl_1}{N}+r_1)\ne0,
        \qquad
        \widehat\mu(\frac{x+tl_2}{N}+r_2)\ne0.
$$
By continuity of $\widehat\mu$ and the set $\mathcal{U}_t$ is compact, there are $\rho_t>0$ and $\delta_t'>0$ such that
$$ |\widehat\mu(\frac{x+tl_1}{N}+y+r_1)|^2\geq\rho_t,\qquad |\widehat\mu(\frac{x+tl_2}{N}+y+r_2)|^2\geq\rho_t
$$
for all $\lvert y\rvert<\delta_t'$. By the decomposition property of self-similar measures, we have
\begin{align*}
    \lvert \widehat{\mu}(x+y+t(l_i+Nr_i)) \rvert^2=&\lvert\widehat{\delta}_{N^{-1}\mathcal{D}}(x+y+tl_i) \rvert^2\cdot|\widehat\mu(\frac{x+tl_i}{N}+y+r_i)|^2\\
    \geq & \rho_t d_t \cdot \frac{1}{n} \sum_{s=0}^{n-1}\lvert  M_{\mathcal{B}s}(x+y)  \rvert^2
\end{align*}
Hence we can take $k_i=l_i+Nr_i$ and $\varepsilon_t=\rho_t d_t>0$.

\end{proof}

\begin{corollary}\label{cor:4.3}
    Let $t \in \mathcal{T}_*$. There exist constants $\epsilon_t>0$ and $\delta_t>0$ such that for every $x\in[0,t]$ and every finite set $F\subset \mathbb{Z}$, one can find distanct integers $r_1,r_2 \notin F$ such that
    $$\lvert \widehat{\mu}(x+tr_i+y) \rvert^2 \geq \epsilon_t \cdot \frac{1}{n} \sum_{s=0}^{n-1} \lvert M_{\mathcal{B}_s}(x+y)\rvert^2, \qquad \lvert y \rvert<\delta_t.$$
\end{corollary}
\begin{proof}
According to Theorem \ref{thm:4.2}, For any $t\in \mathcal{T}_*$, there exist $\varepsilon_t>0$ and $\delta_t>0$, for all $x \in [0,t]$, there exist two distinct integers 
    $$k_1,k_2 \in \mathbb{Z}$$ such that
    $$\lvert \widehat{\mu}(x+y+tk_i) \rvert^2 \geq \varepsilon_t \cdot \frac{1}{n} \sum_{s=0}^{n-1}\lvert  M_{\mathcal{B}s}(x+y)  \rvert^2,$$
    for all $\lvert y \rvert< \delta_t$. We know that
    $$ \lvert \widehat{\mu}_{>n}(x+y+tk_i) \rvert \leq 1 \text{ for all } n\in\mathbb{N}^+,$$
which implies that 
\begin{equation}\label{eq:4.1}
    \lvert \widehat{\mu}_n(x+y+tk_i) \rvert^2 \geq \varepsilon_t \cdot \frac{1}{n} \sum_{s=0}^{n-1}\lvert  M_{\mathcal{B}s}(x+y)  \rvert^2
\end{equation}
 for all  $n\in\mathbb{N}^+$.
Since $\frac{1}{n} \sum_{s=0}^{n-1}\lvert  M_{\mathcal{B}s}(x)  \rvert^2$ is continuous and $\frac{1}{n} \sum_{s=0}^{n-1}\lvert  M_{\mathcal{B}s}(0)  \rvert^2=1$, there exists $\delta'>0$ such that for every $\lvert x\rvert<\delta$, $$\frac{1}{n} \sum_{s=0}^{n-1}\lvert  M_{\mathcal{B}s}(x)  \rvert^2\geq \frac{1}{2}.$$
Let $n$ be sufficiently large such that
$\lvert\frac{x+tk_i}{N^n} \rvert <\delta$, hence 
$$\frac{x+tk_i}{N^n} \in E:=\{ x\in[0,t]:\frac{1}{n} \sum_{s=0}^{n-1} \lvert M_{\mathcal{B}_s} (x)\rvert^2  \geq \frac{1}{2} \} \subset \mathcal{U}. $$
According to Lemma \ref{lem:3.5}, there are two distanct $v_1,v_2 \in \mathbb{Z}$ such that
$$ \lvert \widehat{\mu}(\frac{x+tk_i}{N^n}+tv_1) \rvert \neq 0, \quad \lvert \widehat{\mu}(\frac{x+tk_i}{N^n}+tv_2) \rvert \neq 0.$$
By the compactness of $E$  and the continuity of $\widehat{\mu}$, there exists a $\delta'>0$ such that 
\begin{equation}\label{eq:4.2}
    \lvert \widehat{\mu}(\frac{x+tk_1}{N^n}+tv_i+y) \rvert^2 >\varepsilon_t'
\end{equation}
for all $\lvert y \rvert<\delta'$.
Because both $k$ and the possible $v_i$ range over fixed finite sets, $n$ can be chosen so large that
$$r_i=k_1+N^n v_i \notin F,\qquad i=1,2.$$
Then by \eqref{eq:4.1} and \eqref{eq:4.2}, 
$$ \lvert \widehat{\mu}(x+y+tr_i) \rvert^2=\lvert \widehat{\mu}_n(x+y+tr_i) \rvert^2 \lvert \widehat{\mu}_{>n}(x+y+tr_i) \rvert^2 \geq \varepsilon_t' \varepsilon_t \cdot \frac{1}{n} \sum_{s=0}^{n-1}\lvert  M_{\mathcal{B}s}(x+y)  \rvert^2.$$
Now take $\epsilon_t:=\varepsilon_t' \varepsilon_t>0$, which completes the proof of the theorem. 
\end{proof}
Next, we construct the spectrum. To ensure that each $t\in \mathcal{T}_*$ appears infinitely many times, we introduce the following notation.
Let 
$$\mathcal{T}_* \cap \mathbb{N}=\{ t_1,t_2,\cdots,t_n, \cdots\} $$
and 
$$\mathcal{S}=\{a_k\}_{k=1}^{\infty}:=\{t_1,t_1,t_2,t_1,t_2,t_3,t_1,\cdots \}.$$

Taking $q_0=0$. Let $C_{p_1} \subset \Gamma_{p_1}$ be any finite set containing $0$ and $\lvert N^{-p_1}\xi \rvert<\delta_{a_1}$. Define the set 
\begin{equation*}
    \widetilde{\Lambda}_{q_1}:=\{c_1+N^{p_1} k_{c_1,p_1}:c_1 \in C_{p_1}\}, 
\end{equation*}
where $k_{c_1,p_1}$ is the integer obtained by applying Theorem \ref{thm:4.2}  to point $N^{-q_1}a_1c_1$ for $c_1 \in C_{p_1}$. And
define the set
$$\overline{\Lambda}_{q_1}:=\{e_1+N^{p_1} k_{e_1,p_1}:e_1 \in \Gamma_{p_1} \backslash C_{p_1}\},$$
where $k_{e_1,p_1}$ is the integer obtained by applying Corollary \ref{cor:4.3} to point $N^{-q_1}a_1e_1$ for $e_1 \in \Gamma_{p_1} \backslash C_{p_1}$. 
Moreover, the choice of $k_{e_1,p_1}$ is such that for any two distinct elements $\lambda_1,\lambda_2\in\overline{\Lambda}_{q_1}$ one has 
$$\lvert\lambda_1 \rvert>2 \lvert\lambda_2 \rvert \text{ or } \lvert\lambda_2 \rvert>2 \lvert\lambda_1 \rvert.$$ 
For any $\lambda \in t(\widetilde{\Lambda}_{q_1} \cup \overline{\Lambda}_{q_1})$, $\lambda=a_1d_1+N^{q_1}a_1k_{d_1,p_1}$, where $d_1\in \{c_1,e_1\} \subset \Gamma_{p_1}$. Combining the selection rule of $k_{d_1,p_1}$, Theorem \ref{thm:4.2} and Corollary \ref{cor:4.3}, we obtain that
\begin{align}\label{eq:4.3}
     \lvert \widehat{\mu}_{>q_1} (\xi+\lambda)\rvert^2 =&\lvert \widehat{\mu} (N^{-q_1}\xi+N^{-q_1}a_1d_1+a_1k_{d_1,p_1})\rvert^2\notag\\ 
     \geq &\min\{\varepsilon_{a_1},\epsilon_{a_1}\}\cdot\frac{1}{n} \sum_{s=0}^{n-1}\lvert  M_{\mathcal{B}_s}(N^{-q_1}\xi+N^{-q_1}a_1d_1)  \rvert^2\notag\\
     = &\min\{\varepsilon_{a_1},\epsilon_{a_1}\}\cdot\frac{1}{n} \sum_{s=0}^{n-1}\lvert  M_{N^{-q_1}\mathcal{B}_s}(\xi+\lambda)  \rvert^2.
\end{align}
Assume that the $(n-1)$-th level has been constructed. We now consider the 
$n$-th level. Choose that $p_n$ satisfy 
$$\sup\{\lvert a_nN^{-q_n}\lambda_{n-1}\rvert:\lambda_{n-1} \in \widetilde{\Lambda}_{q_{n-1}} \cup \overline{\Lambda}_{q_{n-1}}\}<\delta_{a_n}.$$
Let $C_{p_n} \subset \Gamma_{p_n}$ be any finite set containing $0$. Define the set 
\begin{equation*}
\widetilde{\Lambda}_{q_n}=\widetilde{\Lambda}_{q_{n-1}}+N^{q_{n-1}}\{c_n+N^{p_n} k_{c_n,p_n}:c_n \in C_{p_n}\}, 
\end{equation*}
where $k_{c_1,p_1}$ is the integer obtained by applying Theorem \ref{thm:4.2}  to point $N^{-p_n}a_nc_n$ for $c_n \in C_{p_n}$. 
And define the set
\begin{align*}
\overline{\Lambda}_{q_n}=&\overline{\Lambda}_{q_{n-1}}+N^{q_{n-1}}\{e_n+N^{p_n} k_{e_n,p_n}:e_n \in \Gamma_{p_n}\}\\
\cup & (\widetilde{\Lambda}_{q_{n-1}}+N^{q_{n-1}}\{e_n'+N^{p_n} k_{e_n',p_n}:e_n' \in \Gamma_{p_n}\backslash C_{p_n}\}),
\end{align*}
where $k_{e_n,p_n}$ and $k_{e_n',p_n}$ are integers obtained by applying Corollary \ref{cor:4.3}  to points $N^{-p_n}a_ne_n$ and $N^{-p_n}a_ne_n'$. Moreover, the choice of $k_{e_n,p_n}, k_{e_n',p_n}$ is such that for any two distinct elements $\lambda_1,\lambda_2\in\overline{\Lambda}_{q_n}$ one has 
$$\lvert\lambda_1 \rvert>2 \lvert\lambda_2 \rvert \text{ or } \lvert\lambda_2 \rvert>2 \lvert\lambda_1 \rvert.$$ 
For any $\lambda_n \in a_n (\overline{\Lambda}_{q_n} \cup \widetilde{\Lambda}_{q_n})$, $$\lambda_n=a_n\lambda_{n-1}+a_nN^{q_{n-1}}d_n+a_nN^{q_n}k_{d_n,p_n}$$
where $d_n\in \{c_n,e_n,e_n'\} \subset \Gamma_{p_n}$. This construction ensures $$\widetilde{\Lambda}_{q_{n-1}}\subset\widetilde{\Lambda}_{q_n},\qquad \overline{\Lambda}_{q_{n-1}}\subset \overline{\Lambda}_{q_n}.$$ 
Combining the selection rule of $k_{d_n,p_n}$, Theorem \ref{thm:4.2} and Corollary \ref{cor:4.3}, we obtain that
\begin{align*}
     \lvert \widehat{\mu}_{>q_n} (\xi+\lambda_n)\rvert^2 =&\lvert \widehat{\mu} (N^{-q_n}\xi+N^{-q_n}a_n\lambda_{n-1}+N^{-q_n}a_nd_n+a_nk_{d_n,p_n})\rvert^2\notag\\ 
     \geq &\min\{\varepsilon_{a_n},\epsilon_{a_n}\}\cdot\frac{1}{n} \sum_{s=0}^{n-1}\lvert  M_{\mathcal{B}_s}(N^{-q_n}\xi+N^{-q_n}a_n\lambda_{n-1}+N^{-p_n}a_nd_n)  \rvert^2\notag\\
     = &\min\{\varepsilon_{a_n},\epsilon_{a_n}\}\cdot\frac{1}{n} \sum_{s=0}^{n-1}\lvert  M_{N^{-q_n}\mathcal{B}_s}(\xi+\lambda_n)  \rvert^2.
\end{align*}
Thus, we obtain two monotone increasing sequences of sets $\{\widetilde{\Lambda}_{q_k}\}_{k=1}^{\infty}$ and $\{\overline{\Lambda}_{q_{k}}\}_{n=1}^{\infty}$. Moreover, for every $k\geq 1$, every $\lambda\in\Lambda_{q_k}:=\widetilde{\Lambda}_{q_k} \cup \overline{\Lambda}_{q_{k}}$, and every $\xi\in[0,1]$,
\begin{align}\label{eq:4.4}
    \left|\widehat{\mu}_{>q_k}(\xi+a_k\lambda)\right|^2
\geq
\frac{c_{a_k}}{n}
\sum_{s=0}^{n-1}
\left|
M_{N^{-q_k}\mathcal B_s}(\xi+a_k\lambda)
\right|^2, \quad c_{a_k}=\min\{\varepsilon_{a_k},\epsilon_{a_k}\}.
\end{align}
Let $$\widetilde{\Lambda}=\bigcup_{k=1}^{\infty} \widetilde{\Lambda}_{q_k}\text{ and }\overline{\Lambda}=\bigcup_{k=1}^{\infty}\overline{\Lambda}_{q_{k}}.$$

\begin{theorem}\label{thm: 4.4}
   Let $t\in \mathcal{T}_*$ and $\Lambda=\frac{1}{N}\mathcal{L}_2+\widetilde{\Lambda} \cup \overline{\Lambda}$. Then $t\Lambda$ is spectrum of $\mu$.
\end{theorem}
\begin{proof}
For each $t\in \mathcal{T}_*$, there exists a subsequence $\{a_{n_k}\} \subset \{a_{k}\}$ such that $a_{n_k}=t$. 
     Combining  Proposition  \ref{prop:2.2} and \eqref{eq:4.4},  
     $$\sum_{\lambda \in \widetilde{\Lambda} \cup \overline{\Lambda}} \lvert \widehat{\mu}(\xi+t\lambda) \rvert^2=\frac{1}{n}\sum_{s=0}^{n-1} \lvert M_{\mathcal{B}_s} (\xi)\rvert^2.$$
It is well known that $\Lambda$ is a spectrum for a prrobability measure $\mu$ if and only if 
$$Q(\xi)=\sum_{\lambda \in \Lambda} \lvert \widehat{\mu}(\xi+\lambda) \rvert^2=1. $$
This yield that 
$$\sum_{l\in\mathcal{L}_2} \lvert M_{\mathcal{B}_s} (\xi+\frac{t}{N}l)\rvert^2
         = 1.$$
Therefore, 
     \begin{align*}
         \sum_{\lambda \in t\Lambda} \lvert \widehat{\mu}(\xi+\lambda) \rvert^2=\sum_{l\in\mathcal{L}_2}\sum_{\lambda \in t\Lambda} \lvert \widehat{\mu}(\xi+\frac{t}{N}l+\lambda) \rvert^2 
         = \sum_{l\in\mathcal{L}_2}\frac{1}{n}\sum_{s=0}^{n-1} \lvert M_{\mathcal{B}_s} (\xi+\frac{t}{N}l)\rvert^2
         = 1.
     \end{align*}
     This implies that $t\Lambda$ is spectrum of $\mu$.
\end{proof}
 
By the translation invariance and countable stability of Beurling dimension, together with the sparsity of the set $ \overline{\Lambda}$, we obtain that $\dim_{Be}(\Lambda)=\dim_{Be}(\widetilde{\Lambda})$.

\begin{theorem}\label{thm:4.5}
    Let $$ 0 \leq s \leq \frac{\log\# \mathcal{D}}{\log N}.$$ The construction above can be carried out so that there are continuum many common spectra $\Lambda$ satisfying $$\dim_{Be} (\Lambda)=s. $$
\end{theorem}
\begin{proof}
    Given $s \in [0, \frac{\log\# \mathcal{D}}{\log N}]$, we can select subsets $$C_{p_n}=C_{p_1+\cdots+p_{n-1}+1}^{'}+NC_{p_1+\cdots+p_{n-1}+2}^{'}+\cdots+N^{p_n-1}C_{p_1+\cdots+p_{n-1}}^{'}\subset \Gamma_{p_n}, $$ 
  where $C_i^{'} \subset \mathcal{L}$ for $i\in\mathbb{N}$, and the following limit exists:
    $$s=\lim_{n \rightarrow \infty} \frac{\log \#C_{1}'\#C_{2}'\cdots\#C_{n}' }{n \log N}.$$
For any $\lambda_n \in \widetilde{\Lambda}_{q_n}$, with the notation introduced in the above construction, there exist $\lambda_{n-1} \in \widetilde{\Lambda}_{q_{n-1}}$, $c_n \in C_{p_n}$ and $ \lvert k_{c_n,p_n} \rvert \leq N^{m_t}$ such that 
$$ \lvert \lambda_n \rvert=\lvert \lambda_{n-1}+N^{q_{n-1}}c_n+N^{q_n}k_{c_n,p_n} \rvert \leq N^{q_n+m_t+1}.$$
This implies that 
\begin{align*}
    \dim_{Be}(\widetilde{\Lambda} )=&\limsup_{h \rightarrow \infty} \sup_{x \in \mathbb{R}} \frac{\log\#(\widetilde{\Lambda} \cap (x-\frac{1}{2}h,x+\frac{1}{2}h))}{\log h}\\
    \geq & \limsup_{n \rightarrow \infty} \frac{\log\#(\widetilde{\Lambda} \cap (-N^{q_n+m_t+1},N^{q_n+m_t+1}))}{\log 2N^{q_n}}\\
    \geq &\limsup_{n \rightarrow \infty}  \frac{\log \#C_{p_1}\#C_{p_2}\cdots\#C_{p_n} }{q_n \log N}=s.
\end{align*}
In the following, we prove $\dim_{Be}(\widetilde{\Lambda} ) \geq s$. Fix $x \in \mathbb{R}$ and consider the counting set 
$$ A_{x}:=(x-\frac{1}{2}h,x+\frac{1}{2}h) \cap \widetilde{\Lambda}. $$
For $N^m \leq h <N^{m+1}$ with $m \geq 1$, each $\lambda \in A_{x}$ has a unique $N$-ary expansion:
$$\lambda=\sum_{j=0}^{\infty} i_j N^j, \quad i_j \in \{0,1,\ldots,N-1\}. $$
It is easy to see that the number of tail sequence 
$$\{i_{m+1}i_{m+2}\cdots:\lambda \in A_x\}$$
is at most two. Now we estimate the number of possible occurrences of the coefficients in this $N$-ary expansion, i.e., $\#\{i_0i_1 \cdots i_m: \lambda \in A_x\}$. The structure of $\lambda$ as 
$$ \lambda=\sum_{k=1}^n(N^{q_{k-1}}c_k+N^{q_k}k_{c_k,p_k})$$
implies modular constraints 
$$ \lambda=i_0+i_1 N+\cdots+i_{q_1-1}N^{q_1-1} = c_1 \pmod{N^{q_1}}$$
This yields the initial bound:
$$\#\{ i_0i_1 \cdots i_{q_1}: \lambda \in A_x \} \leq \#C_{p_1}=\prod_{i=1}^{p_1}\#C_i'$$
By repeatedly applying this method, we can obtain an estimate for 
$$\# A_x \leq 2\prod_{i=1}^{n} \#C_i'$$
 Hence 
\begin{align*}
    \dim_{Be}(\widetilde{\Lambda} )=&\limsup_{h \rightarrow \infty} \sup_{x \in \mathbb{R}} \frac{\log\#(\widetilde{\Lambda} \cap (x-\frac{1}{2}h,x+\frac{1}{2}h))}{\log h}\\
    \leq & \lim_{n \rightarrow \infty} \frac{\log \#C_{1}'\#C_{2}'\cdots\#C_{n}' }{n \log N}=s .
\end{align*}
In summary, we get $\dim_{Be}(\Lambda)=\dim_{Be}(\widetilde{\Lambda} )=s$. Finally, we prove that the family of such sets is uncountable. In our preceding construction, we know that for each $d_n\in \Gamma_{p_n}$, there are two choices for the corresponding $k_{d_n,p_n}$. Now it suffices to, for each level, only the $k_{d_n,p_n}$ associated with some nonzero element $d_n\in \Gamma_{p_n}$ is permitted to have two options; for the remaining elements are each fixed to one value.
   Thus we obtain the sequence of sets $\{\Lambda(I): I \in \{0,1\}^{\infty}\}$. 
   We want to prove that $\Lambda(I) \neq \Lambda(J)$ for any $I\neq J$. 
   Let $n=\min\{k:i_k\neq j_k\}$, i.e., there exists some $d_n$ such that the corresponding $k_{d_n,p_n}, k_{d_n',p_n}$ are distinct. It is without loss of generality to assume that $$k_{d_n,p_n}>k_{d_n,p_n}'.$$
Let $$\lambda*=N^{q_{n-1}}d_n+N^{q_n} k_{d_n,p_n}' \in \Lambda(I)$$
and for any $\lambda \in \Lambda(J)$, 
$$\lambda=d_1'+N^{q_1}k_{d_1',p_1}+N^{q_1}d_2'+N^{q_2}k_{d_2',p_1}+\cdots$$
Since $d_n' \in \Gamma_{p_n}$, then 
$d_i'=0$ for any $i=1,2,\ldots,n-1$ and $d_n=d_n'$. Hence
$$ \lambda-\lambda*=N^{q_n}(k_{d_n',p_n}-k_{d_n,p_n})+ N^{q_{n}} d_{n+1}'+ N^{q_{n+1}}k_{d_{n+1}',p_{n+1}}+\cdots$$
This implies that $\Lambda(I) \neq \Lambda(J)$. Therefore, $\{\Lambda(I): I \in \{0,1\}^{\infty}\}$ has the cardinality of the continuum. 
 
\end{proof}

\begin{proof}[The proof of Theorem \ref{thm:1.7}]

Combining Theorems \ref{thm: 4.4} and \ref{thm:4.5}, we complete the proof of the theorem.

\end{proof}

\section{The proof of Theorem \ref{thm:four-digit-application}}

We keep the notation introduced in Section~1.  Dividing the digit set by its common odd divisor only produces a nonzero spatial dilation of the measure and leaves the spectral-eigenvalue problem unchanged.  Thus, after this harmless normalization, we may write
\[
b=2^\tau\ell_1,\qquad c=a+2^\tau\ell_2,\qquad \gcd(a,b,c)=1.
\]
Set
\[
p_1=\gcd(a,c-b),\qquad p_2=\gcd(c,b-a),\qquad p_3=\gcd(\ell_1,\ell_2).
\]
The integers $p_1,p_2,p_3$ are odd and pairwise coprime: any odd prime dividing two of them would divide $a,b,c$.

\subsection{The dyadic obstruction}

We first identify the lowest branch of the Fourier zero set.

\begin{lemma}\label{lem:four-zero-tower}
For
\[
M_D(\xi)=\frac14\sum_{d\in D}e^{-2\pi i d\xi},
\]
one has
\[
\mathcal Z(M_D)
=
\frac{1}{2p_1}\mathcal O
\cup
\frac{1}{2p_2}\mathcal O
\cup
\frac{1}{2^{\tau+1}p_3}\mathcal O.
\]
Consequently,
\[
\mathcal Z(\widehat\mu)
=
\bigcup_{j\ge1}\left(
\frac{N^j}{2p_1}\mathcal O
\cup
\frac{N^j}{2p_2}\mathcal O
\cup
\frac{N^j}{2^{\tau+1}p_3}\mathcal O
\right).
\]
Its unique lowest dyadic branch is
\[
R_3:=\frac{N}{2^{\tau+1}p_3}\mathcal O,
\]
whose elements have dyadic valuation $\beta-\tau-1$.
\end{lemma}

\begin{proof}
A vanishing sum of four unit complex numbers splits into two opposite pairs.  The three possible pairings give, respectively,
\[
a\xi,\ (c-b)\xi\in\frac12+\mathbb Z,
\qquad
c\xi,\ (b-a)\xi\in\frac12+\mathbb Z,
\]
and
\[
2^\tau\ell_1\xi,\ 2^\tau\ell_2\xi\in\frac12+\mathbb Z.
\]
This yields the asserted formula for $\mathcal Z(M_D)$.  The zero-tower identity from Section~2 then gives the formula for $\mathcal Z(\widehat\mu)$.  Since $N=2^\beta m$ with $m$ odd, the three branch valuations at level $j$ are
\[
\beta j-1,\qquad \beta j-1,\qquad \beta j-\tau-1,
\]
so the minimum occurs exactly on $R_3$.
\end{proof}

The next lemma is the only place where completeness of a spectrum is used.

\begin{lemma}\label{lem:four-primitive-branch}
If $\Gamma$ is a spectrum for $\mu$ and $0\in\Gamma$, then
\[
\Gamma\cap R_3\ne\varnothing.
\]
\end{lemma}

\begin{proof}
Put
\[
R_1=\frac{N}{2p_1}\mathcal O,
\qquad
R_2=\frac{N}{2p_2}\mathcal O.
\]
We first note that an orthogonal set containing $0$ cannot contain simultaneously a point of $R_1\setminus R_2$ and a point of $R_2\setminus R_1$.  Indeed, if
\[
\lambda_1=\frac{N}{2p_1}u_1,
\qquad
\lambda_2=\frac{N}{2p_2}u_2,
\qquad u_1,u_2\in\mathcal O,
\]
then, if $\lambda_1-\lambda_2$ lies in the first branch, the level-one case gives
\[
p_1u_2=p_2(u_1-w),\qquad w\in\mathcal O
\]
an odd--even contradiction, while every higher level gives
\[
\lambda_2=\frac{N}{2p_1}(u_1-N^{j-1}w)\in R_1.
\]
The second branch is symmetric.  If the difference lies in the third branch, then
\[
\frac{u_1}{p_1}-\frac{u_2}{p_2}
=
\frac{N^{j-1}}{2^\tau p_3}w.
\]
The left-hand side has nonnegative dyadic valuation, hence $j\ge k+2$.  Since $N^{j-1}/2^\tau$ is then an integer, clearing denominators gives
\[
p_3(p_2u_1-p_1u_2)
=
p_1p_2\frac{N^{j-1}}{2^\tau}w.
\]
Reducing modulo $p_1$ and using pairwise coprimality gives $p_1\mid u_1$, so $\lambda_1\in R_2$, again a contradiction.  Thus
\[
\Gamma\cap(R_1\cup R_2)\subset R_i
\]
for one $i\in\{1,2\}$.

Assume now that $\Gamma\cap R_3=\varnothing$.  Let
\[
\mu_{>1}=\delta_{N^{-2}D}*\delta_{N^{-3}D}*\cdots
\]
and define
\[
\eta_1=\frac14(\delta_0+\delta_{a/N}-\delta_{b/N}-\delta_{c/N})*\mu_{>1},
\]
\[
\eta_2=\frac14(\delta_0+\delta_{c/N}-\delta_{a/N}-\delta_{b/N})*\mu_{>1}.
\]
Choose $\eta_i$ according to the preceding alternative.  Each first-level cylinder measure is dominated by $4\mu$, hence $\eta_i=f_i\mu$ for a bounded $f_i\in L^2(\mu)$.  The measure $\eta_i$ is nonzero because its Fourier transform is not identically zero.  Since
\[
\widehat\mu_{>1}(\lambda)=\prod_{j\ge2}M_D(\lambda/N^j)=\widehat\mu(\lambda/N),
\]
we have
\[
\widehat\eta_i(\lambda)=P_i(\lambda)\widehat\mu(\lambda/N),
\]
where
\[
P_1(\lambda)=\frac14\left(1+e^{-2\pi i a\lambda/N}-e^{-2\pi i b\lambda/N}-e^{-2\pi i c\lambda/N}\right)
\]
and
\[
P_2(\lambda)=\frac14\left(1+e^{-2\pi i c\lambda/N}-e^{-2\pi i a\lambda/N}-e^{-2\pi i b\lambda/N}\right).
\]
For $\lambda=0$ the corresponding polynomial vanishes; for a zero of level at least two, $\widehat\mu(\lambda/N)=0$; and for a first-level zero in $R_i$, the defining odd congruences give $P_i(\lambda)=0$.  Hence $\widehat\eta_i(\lambda)=0$ for every $\lambda\in\Gamma$.  This contradicts the completeness of $E(\Gamma)$, and proves the lemma.
\end{proof}

We now derive the necessary part of Theorem~\ref{thm:four-digit-application}.  Let $S\subset\mathbb R\setminus\{0\}$ and suppose that $\Gamma\in V(\mu,S)$.  After translating $\Gamma$, assume $0\in\Gamma$, and choose
\[
\lambda_0=\frac{N}{2^{\tau+1}p_3}M\in\Gamma\cap R_3,
\qquad M\in\mathcal O,
\]
by Lemma~\ref{lem:four-primitive-branch}.  Fix $q\in S$.  Since $q\Gamma$ is also a spectrum, $q\lambda_0\in\mathcal Z(\widehat\mu)$, so $q\in\mathbb Q\setminus\{0\}$.  Writing $s=v_2(q)$, where $v_2$ denotes the dyadic valuation on $\mathbb Q\setminus\{0\}$, and comparing the branches in Lemma~\ref{lem:four-zero-tower} gives
\[
s\in\{\beta n:n\ge0\}\cup\{\tau+\beta n:n\ge0\}.
\]
Applying Lemma~\ref{lem:four-primitive-branch} to $q\Gamma$, choose $q\gamma_0\in R_3$ with $\gamma_0\in\Gamma\setminus\{0\}$.  Since $0\in\Gamma$, orthogonality gives $\gamma_0\in\mathcal Z(\widehat\mu)$, and a second comparison yields
\[
s\in\{-\beta n:n\ge0\}\cup\{-\tau-\beta n:n\ge0\}.
\]
Therefore $s=0$.  Hence $q\lambda_0$ lies again in the unique lowest branch $R_3$, and
\[
q\lambda_0=\frac{N}{2^{\tau+1}p_3}M_q
\qquad(M_q\in\mathcal O).
\]
Thus $q=M_q/M$, and consequently
\begin{equation}\label{eq:four-common-necessary}
S\subset M^{-1}\mathcal O.
\end{equation}
In particular, $E_2(\mu)\subset\mathcal O/\mathcal O$.

\subsection{The product-form realization}

It remains to prove that the obstruction above is sharp.  The case $k=0$ is an ordinary Hadamard-triple case: with
\[
L_A=\left\{0,\frac N2\right\},
\qquad
L_B=\left\{0,\frac{N}{2^{r+1}}\right\},
\]
the triple $(N,D,p(L_A\oplus L_B))$ is Hadamard for every odd integer $p$.  Indeed, the difference $pN/2$ pairs $0$ with $a$ and $b$ with $c$, while every other nonzero difference is $pNq/2^{r+1}$ with $q$ odd and pairs $0$ with $b$ and $a$ with $c$.  By Lu's theorem~\cite{Lu 01}, for each admissible $s$ there are continuum many spectra $\Omega$ such that $p\Omega$ is a spectrum for every odd integer $p$.  If $S\subset T^{-1}\mathcal O$, then $\Lambda=T\Omega$ is a spectrum and
\[
q\Lambda=(qT)\Omega
\]
is a spectrum for every $q\in S$, because $qT\in\mathcal O$.  Thus the asserted continuum family follows in this case.

Assume henceforth that $k\ge1$.  Set
\[
D^*=m^kD,
\qquad
Q=N^k,
\qquad
\Delta=D^*+ND^*+\cdots+N^{k-1}D^*.
\]
With $a^*=am^k$, one has
\[
D^*=(\{0\}+QB_0)\cup(a^*+QB_1),
\]
where
\[
B_0=\{0,2^r\ell_1\},
\qquad
B_1=\{0,2^r\ell_2\}.
\]
For $\omega=(\omega_0,\ldots,\omega_{k-1})\in\{0,1\}^k$, put
\[
a_\omega=\sum_{j=0}^{k-1}N^jA_{\omega_j},
\qquad
B_\omega=\sum_{j=0}^{k-1}N^jB_{\omega_j},
\qquad
A_0=0,\ A_1=a^*.
\]
Then
\begin{equation}\label{eq:four-block-product}
\Delta=\bigcup_{\omega\in\{0,1\}^k}(a_\omega+QB_\omega).
\end{equation}
Let
\[
L_A^{[k]}=L_A+NL_A+\cdots+N^{k-1}L_A,
\qquad
L_B^{[k]}=L_B+NL_B+\cdots+N^{k-1}L_B.
\]
Here $A=\{0,a^*\}$ and
\[
A^{[k]}=A+NA+\cdots+N^{k-1}A.
\]
For every odd integer $p$, the two nontrivial one-level phases are
\[
\frac{a^*}{N}\cdot\frac{pN}{2}=\frac{pa^*}{2},
\qquad
\frac{2^r\ell}{N}\cdot\frac{pN}{2^{r+1}}=\frac{p\ell}{2},
\qquad \ell\in\{\ell_1,\ell_2\},
\]
both half-integers.  Hence $(N,\{0,a^*\},pL_A)$ and $(N,B_i,pL_B)$, $i=0,1$, are Hadamard triples.  In the direct-sum triple the mixed frequency differences are also annihilated by the $B_i$ factor, since their numerators modulo $2^{r+1}$ are odd.  Iterating these one-level triples over the $k$ levels shows that
\[
(Q,A^{[k]},pL_A^{[k]}),
\qquad
(Q,B_\omega,pL_B^{[k]}),
\]
and
\[
(Q,A^{[k]}\oplus B_\omega,
 p(L_A^{[k]}\oplus L_B^{[k]}))
\]
are Hadamard triples for every $\omega$.  The tensor-product Hadamard matrices also give the corresponding directness: the sets $A^{[k]}$, $B_\omega$, and $A^{[k]}\oplus B_\omega$ have cardinalities $2^k$, $2^k$, and $4^k$, respectively, and the cosets in \eqref{eq:four-block-product} are disjoint.  Hence \eqref{eq:four-block-product} is a product-form Hadamard triple and every odd integer belongs to its multiplier set.  Therefore $\#\Delta=4^k$, and the $k$-level blocking identity gives
\[
\mu_{N,D^*}=\mu_{Q,\Delta},
\]
and
\[
Q=N^k>4^k=\#\Delta,
\qquad
\frac{\log\#\Delta}{\log Q}=\frac{\log4}{\log N}.
\]

We can now complete the proof.  Let $S\subset T^{-1}\mathcal O$ for some $T\in\mathcal O$, and fix
\[
0\le s\le\frac{\log4}{\log N}.
\]
For $k\ge1$, Theorem~\ref{thm:1.7}, applied to $(Q,\Delta)$, gives continuum many spectra $\Omega$ of $\mu_{Q,\Delta}$ with $\dim_{Be}\Omega=s$ such that $p\Omega$ is a spectrum for every odd integer $p$.  Since $m^kD$ is an odd dilation of $D$, the sets
\[
\Lambda=m^kT\Omega
\]
are spectra for $\mu$, and
\[
q\Lambda=m^k(qT)\Omega
\]
is a spectrum for every $q\in S$.  The Beurling dimension is unchanged by nonzero dilation, and the map $\Omega\mapsto m^kT\Omega$ is injective.  Thus
\[
\#\{\Lambda\in V(\mu,A):\dim_{Be}\Lambda=s\}=2^{\aleph_0}.
\]
The same conclusion for $k=0$ follows from the preceding application of Lu's theorem.

Together with \eqref{eq:four-common-necessary}, this proves the common-set characterization.  Taking $A=\{1,t\}$ gives
\[
E_2(\mu)=\mathcal O/\mathcal O,
\]
and completes the proof of Theorem~\ref{thm:four-digit-application}.\\

\noindent\textbf{Funding.} the second author was supported in part by the National Key Research and Development Program of China (Nos. 2024YFA1013700), NNSF of China (Nos. 12331005).

\bigskip
\begingroup
\parindent=0pt
\small
\textsuperscript{1}\ \paperaddressone
\par\textit{E-mail address:} \paperemailone
\medskip
\par\textsuperscript{2}\ \paperaddresstwo
\par\textit{E-mail address:} \paperemailtwo\quad ($*$ Corresponding author)
\endgroup

\end{document}